\documentclass[journal]{article}
\usepackage{natbib}
\usepackage{amsmath}
\usepackage{amssymb}
\usepackage{multirow}
\usepackage{amsfonts}
\usepackage{mathrsfs}
\usepackage{booktabs}
\usepackage{authblk}
\usepackage{algorithmicx}
\usepackage{algorithm}
\usepackage{booktabs}
\usepackage{algpseudocode}
\usepackage{amsthm}
\usepackage{latexsym}
\usepackage{graphicx}
\usepackage{rotating}
\usepackage{subfigure}
\usepackage{color,varioref}
\usepackage{longtable}
\allowdisplaybreaks[4]
\usepackage[colorlinks,linkcolor=blue]{hyperref}

\newtheorem{property}{Property}
\newtheorem{theorem}{Theorem}

\usepackage{geometry}
\begin{document}
	
\title{Accelerated Tensor Completion via Trace-Regularized Fully-Connected Tensor Network}
	
\author{Wenchao Xie\thanks{School of Mathematics and Computational Science, Xiangtan University, Xiangtan, 411105, China. Email: youmengx@hotmail.com}, \hskip 0.2cm
Qingsong Wang\thanks{School of Mathematics and Computational Science, Xiangtan University, Xiangtan, 411105, China. Email: nothing2wang@hotmail.com}, \hskip 0.2cm
Chengcheng Yan \thanks{School of Mathematics and Computational Science, Xiangtan University, Xiangtan, 411105, China. Email: ycc956176796@gmail.com},  \hskip 0.2cm
Zheng Peng\thanks{Corresponding author.School of Mathematics and Computational Science, Xiangtan University, Xiangtan, 411105, China. Email: pzheng@xtu.edu.cn} 
}
	
\date{}	
\maketitle
\vspace{-15pt}
\begin{abstract}
The fully-connected tensor network (FCTN) decomposition has gained prominence in the field of tensor completion owing to its powerful capacity to capture the low-rank characteristics of tensors. Nevertheless, the recovery of local details in the reconstructed tensor still leaves scope for enhancement. In this paper, we propose efficient tensor completion model that incorporates trace regularization within the FCTN decomposition framework. The trace regularization is constructed based on the mode-$k$ unfolding of the FCTN factors combined with periodically modified negative laplacian. The trace regularization promotes the smoothness of the FCTN factors through discrete second-order derivative penalties, thereby enhancing the continuity and local recovery performance of the reconstructed tensor. To solve the proposed model, we develop an efficient algorithm within the proximal alternating minimization (PAM) framework and theoretically prove its convergence.  To reduce the runtime of the proposed algorithm, we design an intermediate tensor reuse mechanism that can decrease runtime by 10\%–30\% without affecting image recovery, with more significant improvements for larger-scale data. A comprehensive complexity analysis reveals that the mechanism attains a reduced computational complexity. Numerical experiments demonstrate that the proposed method outperforms existing approaches.
\end{abstract}
	
	
\maketitle
	
\section{Introduction}
As a natural high-dimensional generalization of a matrix, the tensor provides a powerful framework for representing complex data structures and has been widely applied in various fields, such as image processing \cite{Cichocki2015Tensor, Du2017PLTD}, hyperspectral data recovery \cite{Li2010Tensor, Xing2012Dictionary}, machine learning \cite{Li2015, Yang2024Spectral, Sidiropoulos2017}, and face recognition \cite{Attouch2013, Hao2013}, among others \cite{He2020, Kilmer2021, Miao2021, Zhang2022Sparse, Zhao2020}. Among its many applications, high-dimensional image completion not only possesses considerable theoretical significance but also offers extensive practical potential across a wide range of fields, e.g \cite{Li2021a, Li2020b, Zhang2021, Zhao2021}. In real-world scenarios, high-dimensional image data are frequently subject to corruption or partial loss due to various factors, such as environmental disturbances, sensor noise, or equipment malfunctions, which can severely compromise the usability and reliability of the data. Tensor completion techniques are designed to recover or reconstruct the full underlying data from partially observed or incomplete measurements, providing a systematic approach to infer missing entries and restore the integrity of high-dimensional datasets.
 Given that high-dimensional images can naturally be represented as tensors, the application of tensor completion methods provides a powerful and effective approach for recovering missing or corrupted information, thereby enhancing the quality, integrity, and applicability of image data in practical tasks. The tensor completion model is usually expressed in the following form
$$
\begin{aligned}
\min_{\mathcal{X}}&\quad\text{rank}(\mathcal{X}),\\\mathrm{s.t.}&\quad\mathcal{P}_{\Omega}(\mathcal{X}-\mathcal{Y})=0,
\end{aligned}
$$
where $\mathcal{X}$ is target tensor, $\mathcal{Y}$ is partially observed tensor, $\Omega$ denotes the index set of the observed elements and $\mathcal{P}_{\Omega}(\mathcal{X})$  is a projection operator that preserves the entries corresponding to the set $\Omega$ while setting all other entries to zero.

Unfortunately, the rank minimization problem is an NP-hard problem, as its objective function rank(·) is discrete and non-convex. Consequently, two types of tensor completion methods, which are low-rank regularization and  based on tensor decomposition  respectively, have been developed on the basis of the low-rank assumption of tensor. Low-rank regularization methods achieve rank minimization by imposing convex surrogate functions of low-rankness on the tensor structure, as exemplified by Zhang et al. \cite{Zhang2014}, who proposed a convex surrogate function to replace direct rank minimization. Tensor decomposition expresses the original tensor as a collection of factor tensors, and by imposing rank constraints on these factors, it effectively achieves rank minimization. Common ones based on tensor decomposition include the CANDECOMP/PARAFAC (CP) decomposition \cite{TG, xie2025efficient}, the Tucker decomposition \cite{Bengua2017Efficient, Liu2013Tensor}, the tensor singular value (SVD) decomposition \cite{Grasedyck2010, Song2020, Zhang2021DeepPrior, Zhang2017Exact}, tensor train (TT) decomposition \cite{oseledets2011tensor, Oseledets2009}, the tensor ring (TR) decomposition \cite{Chen2022FSTRD, Qin2024TRDAnomaly, zhao2016tensor}, and the fully-connected tensor network (FCTN) decomposition \cite{Zheng2022, Zheng2021}. Among these tensor decompositions, tensor train (TT), tensor ring (TR), and the fully-connected tensor network (FCTN) have attracted significant attention in the field of tensor completion. The TT decomposition represents an $N$th-order tensor $\mathcal{X}\in \mathbb{R}^{I_1\times I_2\times \cdots I_N }$ using two matrices $\mathbf{A}_1\in \mathbb{R}^{I_1\times R_1}$ and $\mathbf{A}_N \in \mathbb{R}^{R_{N-1}\times I_N}$ , and $N$-2 third-order tensors $\mathcal{A}_{k} \in \mathbb{R}^{R_{k-1}\times I_{k} \times R_k} $ for $k=2,3,\cdots,N-1$. The corresponding relationship for each element of tensor $\mathcal{X}$ in the TT decomposition is as follows
$$
\mathcal{X}\left(i_{1}, i_{2}, \ldots, i_{N}\right)=\sum_{r_{1}=1}^{R_{1}} \sum_{r_{2}=1}^{R_{2}} \cdots \sum_{r_{N-1}=1}^{R_{N-1}} \mathbf{A}_{1}\left(i_{1}, r_{1}\right) \mathcal{A}_{2}\left(r_{1}, i_{2}, r_{2}\right) \ldots \mathbf{A}_{N}\left(r_{N-1}, i_{N}\right).
$$
The rank of the TT is defined as $\mathbf{r}_{\text{TT}}=(R_{1}, R_2, \cdots, R_{N-1})\in \mathbb{R}^{N-1}$. The TR decomposition represents an $N$th-order tensor using $N$ third-order tensors, and each element of the represented tensor has the following corresponding relationship
$$
\mathcal{X}\left(i_{1}, i_{2}, \ldots, i_{N}\right)=\sum_{r_{1}=1}^{R_{1}} \sum_{r_{2}=1}^{R_{2}} \cdots \sum_{r_{N-1}=1}^{R_{N}} \mathcal{A}_{1}\left(r_{N}, i_{1}, r_1\right) \mathcal{A}_{2}\left(r_{1}, i_{2}, r_{2}\right) \ldots \mathcal{A}_{N}\left(r_{N-1}, i_{N},r_{N}\right),
$$
where $\mathcal{A}_{1}\left(r_{N}, i_{1}, r_1\right)\in \mathbb{R}^{R_N \times I_1 \times R_1}$ and $\mathcal{A}_{k}\left(r_{N}, i_{1}, r_1\right)\in \mathbb{R}^{R_{k-1} \times I_k \times R_k}$ for $k=2,3,\cdots,N$ are TR factors. The rank of the TR is defined as $\mathbf{r}_{\text{TR}}=(R_{1}, R_2, \cdots, R_{N-1})\in \mathbb{R}^{N}$. Both the tensor completion method proposed by Bengua et al. \cite{Bengua2017Efficient} based on TT decomposition and that proposed by Yuan et al. \cite{yuan2019tensor} based on tensor ring have achieved satisfactory results. However, the TT and TR decompositions can only characterize the correlation between two adjacent modes in a tensor and are unable to keep the invariance for tensor transposition. 

To address the limitations of TT and TR decompositions, Zheng et al. \cite{Zheng2021} proposed a fully-connected tensor network (FCTN) decomposition. The FCTN decomposition formulates an  $N$th order tensor $\mathcal{X} \in \mathbb{R}^{I_{1} \times I_{2} \times \cdots \times I_{N}}$ as a multilinear product of  $N$  factor tensors. The element-wise correspondence of the tensor  $\mathcal{X}$  as
$$
\begin{aligned}
\mathcal{X}\left(i_{1}, i_{2}, \ldots, i_{N}\right)= & \sum_{r_{1,2}=1}^{R_{12}} \ldots \sum_{r_{1, N-1}=1}^{R_{1, N-1}} \sum_{r_{1, N}=1}^{R_{1, N}} \sum_{r_{2,3}=1}^{R_{23}} \ldots \sum_{r_{2, N-1}}^{R_{1,N-1}} \sum_{r_{2, M}=1}^{R_{2, N}} \ldots \sum_{r_{N-2, N}=1}^{R_{N-2, M}} \sum_{r_{N-1, N}=1}^{R_{N-1, N}} \\
&\left(\mathcal{A}_{1}\left(i_{1}, r_{1,2}, \ldots, r_{1, N-1}, r_{1, N}\right)\right. \mathcal{A}_{2}\left(r_{1,2}, i_{2}, \ldots, r_{2, N-1}, r_{2, N}\right) \ldots \\
& \mathcal{A}_{k}\left(r_{1, k}, \ldots, r_{k-1, k}, i_{k}, r_{k, k+1}, \ldots, r_{k, N}\right) \ldots \left.\mathcal{A}_{N}\left(r_{1, N}, r_{2, N}, \ldots, r_{N-1, N}, i_{N}\right)\right) ,
\end{aligned}
$$
where $\mathcal{A}_{k} \in \mathbb{R}^{R_{1, k} \times R_{2, k} \times \cdots \times R_{k-1, k} \times I_{k} \times R_{k, k+1} \times \cdots \times R_{k, N}}$ for $k=1,2,\cdots,N$ are FCTN factors. The rank of the FCTN is defined as  $\mathbf{r}_{\text{FCTN}}=(R_{1,2}, R_{1,3}, \ldots, R_{1, N},R_{2,3} , \left.R_{2,4}, \ldots, R_{2, N}, \ldots, R_{N-1, N}\right) \in \mathbb{R}^{\frac{N(N-1)}{2}}$. According to the FCTN decomposition, Zhao et al. \cite{Zheng2022} proposed a tensor completion method that combines FCTN decomposition with a regularization factor, achieving extraordinary results. However, although the model combines global low-rankness with local continuity, it still struggles to recover fine local details. In particular, at low sampling rates, the reconstructed images exhibit noticeable local noise. Moreover, the model fails to fully exploit the intermediate information generated during iterative optimization, which leads to increased computational time.

Therefore, to address the aforementioned issues, this paper proposes an innovative method based on FCTN decomposition, named AFCTNLR. This method incorporates trace regularization to enhance the smoothness of FCTN factors, thereby improving the recovery of local details in the reconstructed tensor. In addition, we fully exploit the intermediate tensors generated during the algorithm’s iterations, effectively reducing the runtime. Overall, this paper makes three main contributions.
\begin{itemize}
    \item  \textbf{Streamlined and effective trace regularized tensor completion model.} We propose a novel trace regularized tensor completion model grounded in the FCTN decomposition framework. This model ensures the global low-rank property of the tensor while enhancing the smoothness of tensor factors via the perturbed laplacian trace regularization, thereby strengthening the continuity of the reconstructed tensor and improving the recovery of local details.
    \item \textbf{Lower computational complexity.} To efficiently solve the proposed model, we developed an algorithm based on the PAM framework, proved its theoretical convergence, and designed an intermediate tensor reuse mechanism to reduce both runtime and computational complexity. Detailed analysis of computational complexity demonstrates that the reuse mechanism consistently lowers the computational cost in each full iteration.
    \item \textbf{Efficient numerical experiments.} We conducted numerical experiments on color videos (including both old and new types) and multi-temporal hyperspectral images. The results demonstrate that the proposed AFCTNLR algorithm consistently achieves the highest PSNR and SSIM values among all comparison methods, while reducing the runtime by 10\%–30\% compared to the FCTNFR algorithm, which attains the second-highest PSNR and SSIM.

\end{itemize}

The rest of this paper is organized as follows. Section \ref{sect_2} establishes the notations, definitions, and properties required for this paper. Section \ref{sect_3} presents the model, algorithm, complexity analysis, and convergence analysis. Section \ref{sect_4} describes our experiments conducted on color videos and hyperspectral images. Section \ref{sect_5} concludes the work presented in this paper.

\vspace{-2pt}
\section{Notations and Preliminaries}\label{sect_2}

\subsection{Notations}
In this paper, we primarily adopt the notation conventions in \cite{TG}, employing bold lowercase letters (e.g., $\mathbf{x}$), bold uppercase letters (e.g., $\mathbf{X}$), and calligraphic letters (e.g., $\mathcal{X}$) to denote vectors, matrices, and tensors, respectively. Given a matrix $\mathbf{X}\in\mathbb{R}^{I\times I}$, we define its trace as $tr(X)=\sum_{i=1}^{I}x_{ii}$, where $x_{ii}$ is the $i$th diagonal element of $\mathbf{X}$. For an $N$th tensor  $\mathcal{X} \in \mathbb{R}^{I_{1}\times \cdots \times I_{N}}$, we use $\mathcal{X}(i_1,i_{2},\ldots,i_{N})$ to denote its $(i_1,i_{2},\ldots,i_{N})$th element, and define its Frobenius norm and $\ell_1$-norm as $\|X\|_F:=$ $( \sum _{i_{1}, i_{2}, \ldots , i_{N}}| \mathcal{X} ( i_{1}, i_{2}, \ldots , i_{N}) | ^{2}) ^{1/ 2}$ and $\| \mathcal{X} \| _{1}: =$ $\sum _{i_{1}, i_{2}, \ldots , i_{N}}|$ ${\mathcal{X}}(i_{1},i_{2},\ldots,i_{N})|$, respectively. Additionally, we denote $\mathcal{A}_k$ as the $k$th  FCTN factor tensor,  $\mathbf{L}_k \in \mathbb{R}^{I_k\times I_k}$ as the second-order negative difference laplacian matrix, and $\alpha$, $\lambda$, $\delta$, $\beta$ as various parameters. $\mathbf{A}_{k(k)}$, which serves as the mode-$k$ unfolding matrix of tensor $\mathcal{A}_k$, is denoted by  $\mathrm{Unfold}{(\mathcal{A}_{k})}$.

\subsection{Preliminaries}

\textbf{Definition 1} (Generalized Tensor Transposition \cite{Zheng2021}). Given an $N$th tensor $\mathcal{X}\in\mathbb{R}^{I_1\times I_2\times\cdots\times I_N}$  and $\mathbf{n}=(n_1,n_2,\ldots,n_N)$ be an arbitrary rearrangement of $(1,2,\ldots,N)$. The tensor obtained by permuting the modes of $\mathcal{X}$ according to the order given by $\mathbf{n}$ is referred to as the $n$-mode transposition of $\mathcal{X}$, denoted by $\vec{ \mathcal{X}}^{\mathbf{n}}\in\mathbb{R}^{I_{n_{1}}\times I_{n_{2}}\times\cdots\times I_{n_{N}}}$. In other words, $\vec{\mathcal{X}^{\mathbf{n}}}(i_{n_{1}},i_{n_{2}},\ldots,i_{n_{N}})$ = $\mathcal{X}(i_{1},i_{2},\ldots,i_{N})$ for $\forall i_{k}=1,2\cdots,I_{k}$, where $k=1,2,\cdots,N$.\\

\noindent\textbf{Definition 2} (Generalized Tensor Unfolding  \cite{Zheng2021}). Given an $N$th tensor $\mathcal{X}\in\mathbb{R}^{I_1\times I_2\times\cdots\times I_N}$  and $\mathbf{n}=(n_1,n_2,\ldots,n_N)$ be an arbitrary rearrangement of $(1,2,\ldots,N)$. We define generalized tensor unfolding of $\mathcal{X}$ as an $\prod_{i=1}^dI_{n_i}\times\prod_{i=d+1}^NI_{n_i}$ matrix $\mathbf{X}_[n_{1:d}|n_{d+1:N}]$, whose elements satisfy
$\mathbf{X}_{[n_{1:d}|n_{d+1:N}]}(j_{1},j_{2})=\mathcal{X}(i_{1},i_{2},\ldots,i_{N})$
with
$$j_{1}=i_{n_{1}}+\sum_{s=2}^{d}\left((i_{n_{s}}-1)\prod_{m=1}^{s-1}I_{n_{m}}\right)$$ and  $$j_{2}=i_{n_{d+1}}+\sum_{s=d+2}^{N}\left((i_{n_{s}}-1)\prod_{m=d+1}^{s-1}I_{n_{m}}\right).$$ 
We employ $\mathbf{X}_{[n_{1:d}|n_{d+1:N}]}={\mathrm{GUnfold}(\mathcal{X},n_{1:d}|n_{d+1:N})}$ and $\mathcal{X}=\mathrm{GFold}(\mathbf{X}_{[n_{1:d}|n_{d+1:N}]},n_{1:d}|n_{d+1:N})$ to represent the unfolding operation along with its inverse.
\\

\noindent\textbf{Definition 3} (Tensor Contraction \cite{Zheng2021}). Let tensors $\mathcal{X} \in \mathbb{R}^{l_1 \times l_2 \times \cdots \times l_N}$ and $\mathcal{Y} \in \mathbb{R}^{l_1 \times l_2 \times \cdots \times l_M}$ satisfy $I_{n_i} = J_{m_i}$ for $i = 1, 2, \ldots, d$. Here, $\mathbf{n} = (n_1, n_2, \ldots, n_N)$ and $\mathbf{m} = (m_1, m_2, \ldots, m_M)$ are arbitrary permutations of $(1, 2, \ldots, N)$ and $(1, 2, \ldots, M)$ such that $n_{d+1} < n_{d+2} < \cdots < n_N$ and $m_{d+1} < m_{d+2} < \cdots < m_M$, respectively. The contraction between the $n_{1:d}$th-modes of $\mathcal{X}$ and the $m_{1:d}$th-modes of $\mathcal{Y}$, denoted by $\mathcal{X} \times_{n_{1:d}}^{m_{1:d}} \mathcal{Y}$, produces a tensor of $(N+M-2d)$ order
\[
\mathcal{Z} = \mathcal{X} \times_{n_{1:d}}^{m_{1:d}} \mathcal{Y} \in \mathbb{R}^{l_{n_{d+1}} \times \cdots \times l_{n_N} \times l_{m_{d+1}} \times \cdots \times l_{m_M}}.
\]
The tensor $\mathcal{Z}$ satisfies the following relationship:
$$
\mathbf{Z}_{[1:M-d;M-d+1:M+N-2d]}=\mathbf{X}_{[n_{d+1:N};n_{1:d}]}\mathbf{Y}_{[m_{1:d};m_{d+1:M}]}.
$$

\begin{property}\cite{Zheng2021}\label{pro_1}
 Let $\mathcal{X}=\operatorname{FCTN}\left(\left\{\mathcal{A}_{k}\right\}_{k=1}^{N}\right)$ denote that the tensor $\mathcal{X} \in \mathbb{R}^{I_{1} \times I_{2} \times \ldots \times I_{N}}$ is formed from a collection of factors $\mathcal{A}_{k}$ $(k=1,2,\ldots, N)$, and  $\mathcal{M}_{k}=\operatorname{FCTN}\left(\left\{\mathcal{A}_{k}\right\}_{k=1}^{N}, / \mathcal{A}_{k}\right)$ signifies that $\mathcal{A}_{k}$ $(k \in{1,2, \ldots, N})$ is omitted from the calculation. Consequently, the following property is satisfied 
 $$\mathbf{X}_{(k)}=\mathbf{A}_{k(k)} \mathbf{M}_{k\left[i_{1: N-1};n_{1: N-1}\right]},$$
where
$$i_{t}=\left\{\begin{array}{l}
2 t, \quad \text { if } t<k, \\
2 t-1, \text { if } t \geq k,
\end{array} \quad \text { and }\quad  n_{t}=\left\{\begin{array}{ll}
2 t-1, & \text { if } t<k, \\
2 t, & \text { if } t \geq k .
\end{array}\right.\right.$$\\
\end{property}

\begin{theorem}
(Fast solution of the Sylvester matrix equation \cite{Sylvester1867, Zheng2022})\label{TH_1}
Suppose that $\mathbf{M} \in \mathbb{R}^{m \times m}, \mathbf{N} \in \mathbb{R}^{n \times n}$ and $\mathbf{X}, \mathbf{Y} \in \mathbb{R}^{m \times n}$. If  $\mathbf{I}_{n} \otimes \mathbf{M}+\mathbf{N}^{T} \otimes \mathbf{I}_{m}$ is an invertible matrix, the classical Sylvester matrix equation
$$
\mathbf{M X}+\mathbf{X} \mathbf{N}=\mathbf{Y}
$$
has a unique solution. In addition, if  $\mathbf{M}$  and  $\mathbf{N}$  can be decomposed into

$$\mathbf{M}=\mathbf{F} \Psi \mathbf{F}^{T} \text { and }\quad  \mathbf{N}=\mathbf{U} \Sigma \mathbf{U}^{T},$$
where  $\Psi$  and  $\Sigma$  are the diagonal matrices,  $\mathbf{F}$  and  $\mathbf{U}$  are the unitary matrices. We can get a unique solution to  $\mathbf{X}$
$$\mathbf{X}=\mathbf{F}\left((1 \oslash \mathbf{T}) \odot\left(\mathbf{F}^{T} \mathbf{Y U}\right)\right) \mathbf{U}^{T}.$$
Here $\oslash$ and $\odot$ represent element-wise multiplication and element-wise division, respectively.
Assuming that $\operatorname{diag}(\Psi)$ represents the column vector whose elements correspond to the diagonal entries of $\Psi$, $\mathbf{T} \in \mathbb{R}^{m\times n}$ is defined as
$$
\begin{aligned}
\mathbf{T}= & (\operatorname{diag}(\Psi), \operatorname{diag}(\Psi), \ldots, \operatorname{diag}(\Psi)) \\
& +(\operatorname{diag}(\Sigma), \operatorname{diag}(\Sigma), \ldots, \operatorname{diag}(\Sigma))^{T}.
\end{aligned}
$$
\end{theorem} 


\section{Model and Algorithm}\label{sect_3}
In this section, we propose a novel trace regularized tensor completion model based on the FCTN decomposition. Building upon the PAM framework, we design an efficient algorithm to solve the model, in which an intermediate tensor reuse strategy is introduced to improve computational efficiency. Moreover, we conduct a detailed analysis of the computational complexity and convergence of the proposed algorithm.
\subsection{The proposed model}
Assuming that  $\mathcal{Y} \in \mathbb{R}^{I_{1} \times I_{2} \times \cdots \times I_{N}}$  represents a partial observation of the target tensor  $\mathcal{X} \in \mathbb{R}^{I_{1} \times I_{2} \times \cdots \times I_{N}}$, we present the proposed tensor completion model as
\begin{equation}
\begin{aligned}\label{model_1}
\min_{\mathcal{A}_{k},\mathcal{X}}&\quad\frac{1}{2}\left\|\mathcal{X}-\mathrm{FCTN}\left(\{\mathcal{A}_{k}\}_{k=1}^{N}\right)\right\|_{F}^{2}+\sum_{k=1}^{N}\frac{\lambda_{k}}{2}\mathrm{tr}(\mathbf{A}^T_{k(k)}\mathbf{L}_k\mathbf{A}_{k(k)}),\\\mathrm{s.t.}&\quad\mathcal{P}_{\Omega}(\mathcal{X}-\mathcal{Y})=0, k=1,2,\cdots,N,
\end{aligned}
\end{equation}
where  $\mathcal{A}_{k} \in \mathbb{R}^{R_{1, k} \times R_{2, k} \times \cdots \times R_{k-1, k} \times I_{k} \times R_{k, k+1} \times \cdots \times R_{k, N}}$, $\mathbf{A}_{k(k)} \in \mathbb{R}^{I_{k} \times \prod_{i=1, i \neq k}^{N} R_{k, i}}$  is the mode-$k$  unfolding of  $\mathcal{A}_{k}$, $\lambda_{k}>0$  is a regularization parameter. $\mathbf{L}_k \in \mathbb{R}^{I_k \times I_k}$ of the periodically modified negative laplacian matrix is given as follows
$$
\mathbf{L_k}= \begin{bmatrix}
  -2-\delta & 1 &  0& \cdots & 0&1\\
  1& -2-\delta &  1&  \cdots& 0&0\\
  0&  1&  -2-\delta &  \cdots&0&0 \\
\vdots&\vdots&\vdots&\ddots&\vdots &\vdots  \\
0&0&0&\cdots&-2-\delta &1\\
  1&  0&  0&\cdots& 1&-2-\delta 
\end{bmatrix}, 
$$
where $\delta>0$ is a hyperparameter. $\Omega$ denotes the index set of the observed elements, and 
$\mathcal{P}_{\Omega}(\mathcal{X})$ is the projection operator that preserves the entries in 
$\Omega$ while setting all other entries to zero.

In our proposed model, the objective consists of two components, $\left\|\mathcal{X}-\mathrm{FCTN}\left(\{\mathcal{A}_{k}\}_{k=1}^{N}\right)\right\|_{F}^{2}$ and $\mathrm{tr}(\mathcal{A}^T_{k(k)}\mathbf{L}_k\mathcal{A}_{k(k)})$. During the optimization process, each component is converted into a subproblem, and the corresponding FCTN factors are solved accordingly. 
The $\left\|\mathcal{X}-\mathrm{FCTN}\left(\{\mathcal{A}_{k}\}_{k=1}^{N}\right)\right\|_{F}^{2}$ serves to maintain the low-rank property of the reconstructed tensor $\mathcal{X}$. 
According to Property \ref{pro_1}, every column of $\mathbf{X}_{(k)}$ can be represented as a linear combination of the columns of $\mathbf{A}_{k(k)}$. Therefore, the $\mathrm{tr}(\mathcal{A}^T_{k(k)}\mathbf{L}_k\mathcal{A}_{k(k)})$ penalizes the second-order variations of $\mathbf{A}_{k(k)}$, thereby enhancing the smoothness of $\mathbf{A}_{k(k)}$ and indirectly improving the local smoothness of $\mathbf{X}_{(k)}$, which in turn boosts the quality of tensor completion.

\subsection{The proposed algorithm}
In this subsection, we employ the PAM algorithm framework to solve our proposed model. The problem (\ref{model_1}) is converted into the unconstrained optimization problem as follows:
\begin{equation}
\begin{aligned}\label{model_2}
\min_{\mathcal{A}_{k},\mathcal{X}}&\quad\frac{1}{2}\left\|\mathcal{X}-\mathrm{FCTN}\left(\{\mathcal{A}_{k}\}_{k=1}^{N}\right)\right\|_{F}^{2}+\sum_{k=1}^{N}\frac{\lambda_{k}}{2}\mathrm{tr}(\mathbf{A}^T_{k(k)}\mathbf{L}_k\mathbf{A}_{k(k)})+\iota_{\mathbb{D}}(\mathcal{X}),\\\mathrm{where}&\quad\iota_{\mathbb{D}}(\mathcal{X}):=\left\{\begin{array}{c}
0, \text { if } \mathcal{X} \in \mathbb{D}, \\
\infty, \text { otherwise. }
\end{array} \text { with } \mathbb{D}:=\left\{\mathcal{X}: \mathcal{P}_{\Omega}(\mathcal{X}-\mathcal{Y})=0\right\}\right., k=1,2,\cdots,N.
\end{aligned}
\end{equation}
Through the PAM algorithm framework, we can optimize the optimization problem in (\ref{model_2}) by alternately updating the following system of equations:
\vspace{8pt}
\begin{equation}
\left\{\begin{array}{l}
\begin{aligned}\label{model_3}
\mathcal{A}_{k}^{(t+1)}= \underset{\mathcal{A}_{k}}{\arg \min }\,\,& \frac{1}{2} \| \mathcal{X}^{(t)}-\operatorname{FCTN}\left(\mathcal{A}_{1, k-1}^{(t+1)}, \mathcal{A}_{k}, \mathcal{A}_{k+1: N}^{(t)}) \|_{F}^{2}\right. \\
&+\frac{\lambda_{k}}{2}\mathrm{tr}(\mathcal{A}^T_{k(k)}\mathbf{L}_k\mathcal{A}_{k(k)})+\frac{\rho}{2}\left\|\mathcal{A}_{k}-\mathcal{A}_{k}^{(t)}\right\|_{F}^{2}, \\
\mathcal{X}^{(t+1)}=\underset{\mathcal{X}}{\arg \min }\,\,& \frac{1}{2}\left\|\mathcal{X}-\operatorname{FCTN}\left(\left\{\mathcal{A}_{k}^{(t+1)}\right\}_{k=1}^{N}\right)\right\|_{F}^{2}+\iota_{\mathbb{S}}(\mathcal{X})+\frac{\rho}{2}\left\|\mathcal{X}-\mathcal{X}^{(t)}\right\|_{F}^{2},
\end{aligned}
\end{array}\right.
\end{equation}
where $k=1,2,\cdots,N,$ $\rho>0$ is a proximal parameter, and $t$ is the iteration index. We now turn to the update rules of each subproblem. For Equations (\ref{model_3}), the updating procedure involves two main steps: (i) updating the FCTN factor $\mathcal{A}_k^{(t+1)}$, and (ii) updating the reconstructed tensor $\mathcal{X}^{(t+1)}$.

For updating the FCTN factor $\mathcal{A}_k^{(t+1)}$, since its derivative is difficult to compute in tensor form, we convert it into the following matrix form:

\begin{equation}\label{model_4}
 \underset{\mathbf{A}_{k(k)}}{\arg \min }\,\,\frac{1}{2} \| \mathbf{X}^{(t)}_{(k)}-\mathbf{A}_{k(k)}\mathbf{M}_{k}^{(t)}\|_{F}^{2}+\frac{\lambda_{k}}{2}\mathrm{tr}(\mathbf{A}^T_{k(k)}\mathbf{L}_k\mathbf{A}_{k(k)})+\frac{\rho}{2}\left\|\mathbf{A}_{k(k)}-\mathbf{A}_{k(k)}^{(t)}\right\|_{F}^{2},
\end{equation}
where $\mathbf{M}_{k}^{(t)}={\mathrm{GUnfold}(\mathcal{M}_k^{(t)},\mathbf{i}|\mathbf{n})}$, $\mathcal{M}^{(t)}_{k}=\operatorname{FCTN}\left(\mathcal{A}^{(t+1)}_{1:k-1},\mathcal{A}_{k},\mathcal{A}^{(t)}_{k+1:N}, /\mathcal{A}_{k}\right)$ and vectors 
$\mathbf{i}$ and $\mathbf{n}$ are set in the same way as in Property \ref{pro_1}. In the matrix form, the differentiation of objective function (\ref{model_4}) is straightforward, and by computing it we obtain the following sylvester matrix equation:
\begin{equation}\label{model_5}
\mathbf{A}_{k(k)}\mathbf{M}_{k}^{(t)}(\mathbf{M}^{(t)}_{k})^T+\lambda_k\mathbf{L}_k\mathbf{A}_{k(k)}+\rho\mathbf{A}_{k(k)}=\mathbf{X}_{(k)}^{(t)}(\mathbf{M}^{(t)}_{k})^T+\rho\mathbf{A}^{(t)}_{k(k)}.
\end{equation}

\noindent By exploiting the properties of the kronecker product, we can transform the sylvester matrix equation (\ref{model_5}) into the following form:
\begin{equation}\label{model_6}
(\lambda_k\mathbf{I}_{s_k}\otimes\mathbf{L}_k+\mathbf{M}_{k}^{(t)}(\mathbf{M}^{(t)}_{k})^T\otimes\mathbf{I}_{q_k}+\rho\mathbf{I}_{s_kq_k})vec(\mathbf{A}_{k(k)})=vec(\mathbf{X}_{(k)}^{(t)}(\mathbf{M}^{(t)}_{k})^T+\rho\mathbf{A}^{(t)}_{k(k)}),
\end{equation}
where $s_k={\prod_{i=1, i \neq k}^{N} R_{k, i}}$ and $q_k=I_k$. According to Theorem \ref{TH_1}, we can solve Equation (6) both accurately and efficiently. Since matrix 
$\mathbf{M}_{k}^{(t)}(\mathbf{M}^{(t)}_{k})^T$ is a symmetric positive semi-definite matrix and $\mathbf{L}_k$ is a periodic circulant matrix, 
$\mathbf{M}_{k}^{(t)}(\mathbf{M}^{(t)}_{k})^T$ can be transformed into a form involving a diagonal matrix via eigenvalue decomposition, while 
$\mathbf{L}_k$ can be transformed into a form involving a diagonal matrix through one-dimensional Fourier transform. Their specific forms are as follows:
\begin{equation}
\mathbf{M}_{k}^{(t)}(\mathbf{M}^{(t)}_{k})^T=\mathbf{C}\Phi \mathbf{C}^T,  \mathbf{L}_k=\mathbf{F}^{H}\Lambda \mathbf{F}.
\end{equation}
Then we can derive a solution as follows:
\begin{equation}\label{eq_8}
\mathbf{A}_{k(k)}^{(t+1)}=\mathbf{F}^{H}\left((1 \oslash \mathbf{T}) \odot\left(\mathbf{F} \mathbf{Y C}\right)\right) \mathbf{C}^{T},    
\end{equation}
where $\mathbf{T}=\lambda_k \text{repvec}(\text{diag}(\Lambda),s_k)+(\text{repvec}(\text{diag}(\Phi),I_k))^T+\rho\text{ones}(I_k,s_k)$, $\mathbf{Y}=\mathbf{X}_{k}^{(t)}(\mathbf{M}^{(t)}_{k})^T+\rho\mathbf{A}^{(t)}_{k(k)}$. The revec($\cdot$) produces an $n\times m$ matrix, which is generated by repeating an $n\times 1$ column vector 
$m$ times. The diag($\cdot$) returns a column vector whose elements are derived from the main diagonal of the diagonal matrix. Therefore, by applying generalized folding to the solution 
$\mathbf{A}_{k(k)}^{(t+1)}$, we can obtain the FCTN factor 
$\mathcal{A}_{k(k)}^{(t+1)}$ as follows: 
\begin{equation}\label{eq_9}
\mathcal{A}_{k}^{(t+1)}= \mathrm{GFold}(\mathbf{A}_{k(k)}^{(t+1)},k|1,\cdots, k-1,k+1,\cdots, N).  
\end{equation}

  For updating the reconstructed tensor $\mathcal{X}^{(t+1)}$, we can find it  closed-form solution, whose form is as follows:
\begin{equation}\label{eq_10}
\mathcal{X}^{(t+1)}=\mathcal{P}_{\Omega ^{c}}\left( \frac{\operatorname{FCTN}\left(\left\{\mathcal{A}_{k}^{(t+1)}\right\}_{k=1}^{N}\right)+\rho\mathcal{X}^{(t)}}{1+\rho}\right)+ \mathcal{P}_{\Omega}(\mathcal{Y}).
\end{equation}

Thus, we obtain Algorithm \ref{FCTNNLR} for solving the proposed model based on the PAM algorithm framework.
\vspace{8pt}
\begin{algorithm}[H]
\caption{PAM-FCTNLR}\label{FCTNNLR}
{\bfseries Input:} The observed tensor $\mathcal{Y}\in \mathbb{R}^{I_{1}\times \cdots \times I_{N}}$, the index of the known elements $\Omega$, the maximal FCTN-rank $r_{\text{FCTN}}^{max}$, and the parameters $\delta_k$ and $\lambda_k$ for $k=1,2,\cdots,N$, $\rho=0.1$, stopping criteria $\epsilon = 1\times10^{-4}$ and maximum number of iterations M.\\
{\bfseries Initialize:} The initial iteration index $t=1$, the initial FCTN-rank $\mathbf{r}_{\text{FNTC}}=\text{ones}(1,N(N-2)/2)$, $\mathcal{X}^{(1)}=\mathcal{Y}$, $\mathcal{A}_{k}^{(1)}=\text{rand}(R_{1,k},R_{2,k},\cdots,R_{k-1,k},I_{k},R_{k,k+1},\cdots,R_{k,N})$, periodically modified negative laplacian matrixs $\mathbf{L}_{k}$. Here $k=1,2,\cdots,N$ and elements generated by $\text{rand}(\cdot)$ follow a standard normal distribution.
\begin{algorithmic}[1]
\While{($\left \| \mathcal{X}^{(t+1)}-  \mathcal{X}^{(t)}\right \| _{F}/\left \| \mathcal{X}^{(t)} \right \|_{F}>\epsilon $) \textbf{ and } ($t =1,2,\cdots, M$)}
\For {$k=1,2,\dots N$}
 \State Obtain $\mathbf{X}_{(k)}^{(t)}$ via $\mathrm{Unfold}{(\mathcal{X}^{(t)})}$.
 \State Compute $\mathcal{M}_{k}^{(t)}=\operatorname{FCTN}\left(\mathcal{A}^{(t+1)}_{1:k-1},\mathcal{A}_{k},\mathcal{A}^{(t)}_{k+1:N}, /\mathcal{A}_{k}\right)$.
 \State Obtain $\mathbf{M}_k^{(t)}$ via ${\mathrm{GUnfold}(\mathcal{M}_k^{(t)},\mathbf{i}|\mathbf{n})}$.
 \State Compute $\mathbf{X}_{k}^{(t)}(\mathbf{M}^{(t)}_{k})^T$, $\operatorname{SVD}(\mathbf{M}^{(t)}_{k}(\mathbf{M}^{(t)}_{k})^T)$ and $\operatorname{FFT}(\mathbf{L}_k)$.
 \State Update $\mathcal{A}_k^{(t+1)}$ according to (\ref{eq_8}) and (\ref{eq_9}).
\EndFor
   \State  Update $\mathcal{X}^{(t+1)}$ via (\ref{eq_10}).
\EndWhile
\end{algorithmic}
{\bfseries Output:} The reconstructed tensor $\mathcal{X}$.
\end{algorithm}
\vspace{8pt}

\subsection{Intermediate tensor reutilization mechanism}

In Algorithm \ref{FCTNNLR}, the main computational cost arises from Lines 5, 7, and 10. For Line 7, apart from using random sampling or other stochastic methods, no alternative approaches are currently available. However, for Lines 5 and 10, the computational burden can be reduced by reusing the intermediate tensors generated during the process.

We illustrate the mechanism of reutilizing intermediate tensors by means of the FCTN decomposition of fourth-order tensors. The FCTN decomposition process for fourth-order tensors is presented in Figure \ref{fig_1}.

\begin{figure}[!ht]
\vspace{-10pt}
    \centering
    \includegraphics[width=0.4\linewidth]{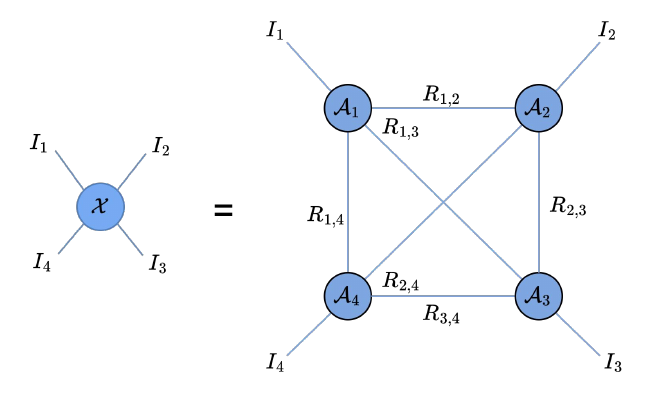}
    \caption{The fourth-order FCTN decomposition of  tensor.}
    \label{fig_1}
\end{figure}

\noindent In the iterative process of Algorithm \ref{FCTNNLR} for a fourth-order tensor, four distinct tensors, $\mathcal{M}_1$, $\mathcal{M}_2$, $\mathcal{M}_3$ ,$\mathcal{M}_4$ are produced. Notably, $\mathcal{M}_1$ and $\mathcal{M}_2$, as well as $\mathcal{M}_3$ and $\mathcal{M}_4$, involve overlapping intermediate tensors, leading to repeated computations. The intermediate tensor shared by $\mathcal{M}_1$ and $\mathcal{M}_2$ is shown in  Figure \ref{fig_2}.

\begin{figure}[!ht]
\vspace{-8pt}
    \centering
    \includegraphics[width=0.70\linewidth]{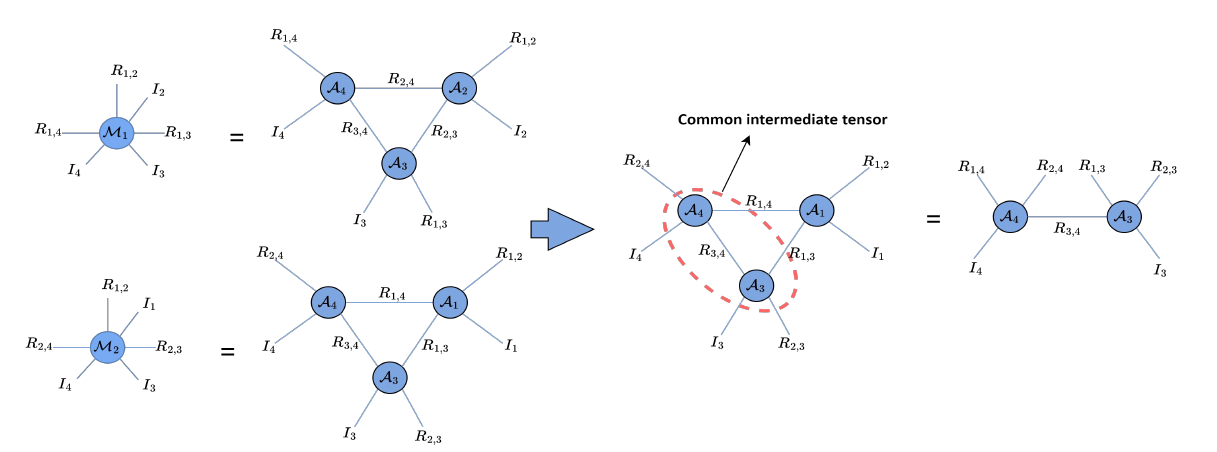}
    \caption{Common intermediate tensor of $\mathcal{M}_1$ and $\mathcal{M}_2$.}
    \label{fig_2}
\end{figure}

\noindent The intermediate tensor common to $\mathcal{M}_3$ and $\mathcal{M}_4$ is depicted in Figure \ref{fig_3}.
\begin{figure}[!ht]
\vspace{-10pt}
    \centering
    \includegraphics[width=0.70\linewidth]{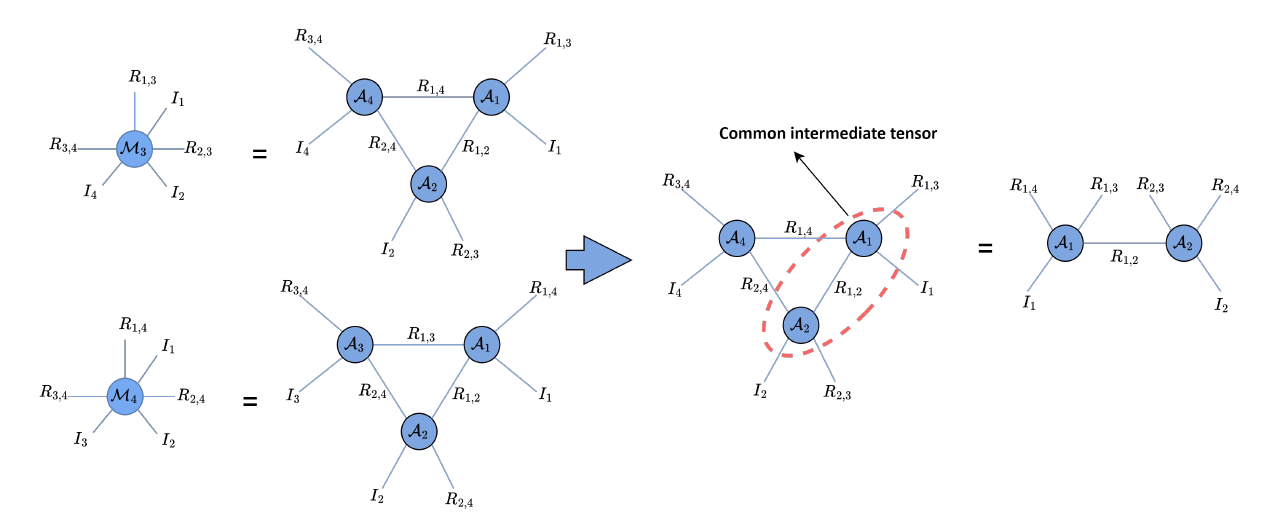}
    \caption{Common intermediate tensor of $\mathcal{M}_3$ and $\mathcal{M}_4$.}
    \label{fig_3}
\end{figure}

\noindent Therefore, during the iteration process, we can store these redundantly computed tensors to avoid repeated calculations, thereby improving the computational efficiency of Line 5 in Algorithm \ref{FCTNNLR}.

For Line 10 in Algorithm \ref{FCTNNLR}, the main computational cost comes from the generation of the $\operatorname{FCTN}\left(\left\{\mathcal{A}_{k}^{(t+1)}\right\}_{k=1}^{N}\right)$. Similarly, we can accelerate this process by utilizing intermediate tensor produced during the algorithm's iteration. Then, we take a fourth-order tensor as an example to demonstrate this acceleration process. The generation of $\operatorname{FCTN}\left(\left\{\mathcal{A}_{k}^{(t+1)}\right\}_{k=1}^{4}\right)$ involves $\mathcal{M}_{4}^{(t)}$, which are shown in Figure \ref{fig_4}.

\begin{figure}[!ht]
    \centering
    \includegraphics[width=0.88\linewidth]{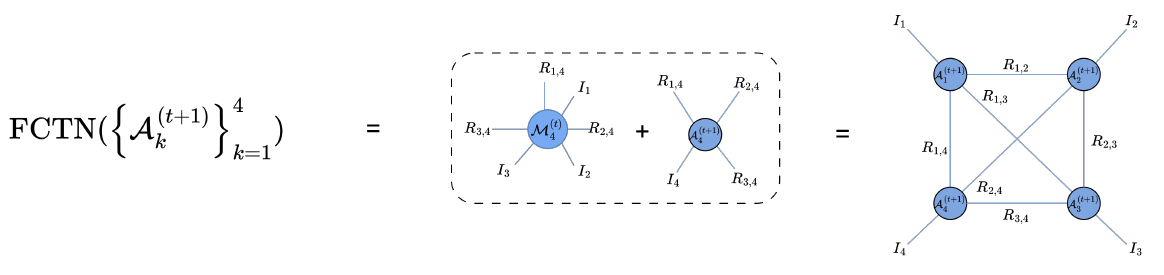}
    \caption{Generate $\operatorname{FCTN}\left(\left\{\mathcal{A}_{k}^{(t+1)}\right\}_{k=1}^{4}\right)$ using $\mathcal{M}_{4}^{(t)}$.}
    \label{fig_4}
\end{figure}

\noindent Within the iteration of Algorithm \ref{FCTNNLR}, $\mathbf{M}_4^{(t)}$ is preserved, and generating  $\operatorname{FCTN}\left(\left\{\mathcal{A}_{k}^{(t+1)}\right\}_{k=1}^{4}\right)$ by applying its generalized unfolding $\mathcal{M}_4^{(t)}$ together with $\mathcal{A}_4^{(t+1)}$ can significantly accelerate the computation in Line 9. Generating $\operatorname{FCTN}\left(\left\{\mathcal{A}_{k}^{(t+1)}\right\}_{k=1}^{4}\right)$ also involves the following different combinations, as shown in Figure \ref{fig_5}.

\begin{figure}[!ht]
    \centering
    \includegraphics[width=0.90\linewidth]{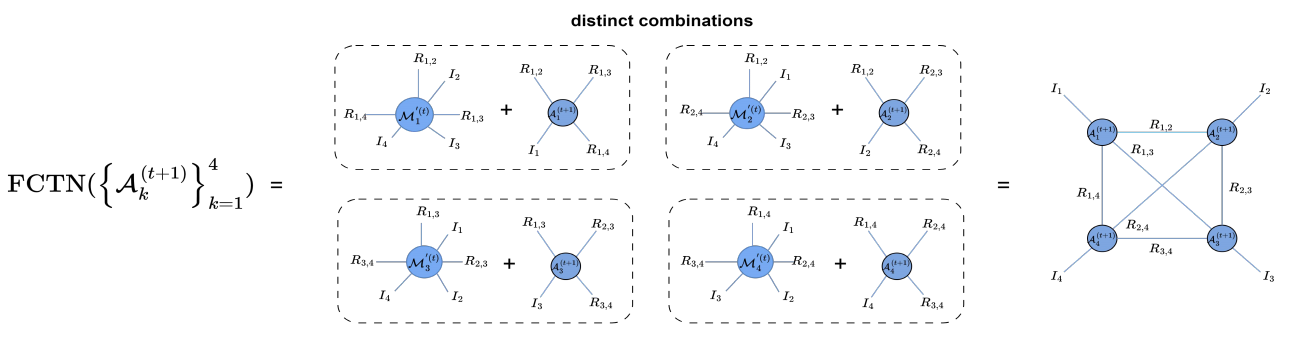}
    \caption{Distinct combinations for generating $\operatorname{FCTN}\left(\left\{\mathcal{A}_{k}^{(t+1)}\right\}_{k=1}^{4}\right)$.}
    \label{fig_5}
\end{figure}

\noindent In Figure \ref{fig_5} $\mathcal{M}^{'(t)}_{k}=\operatorname{FCTN}\left(\mathcal{A}^{(t+1)}_{1:k-1,k+1:N},\mathcal{A}_{k}, /\mathcal{A}_{k}\right)$. The combinations of 
$\operatorname{FCTN}\left(\left\{\mathcal{A}_{k}^{(t+1)}\right\}_{k=1}^{4}\right)$
from Figure \ref{fig_5}, we know that the Algorithm \ref{FCTNNLR}  completes a full iteration under different indices and quickly generates 
$\operatorname{FCTN}\left(\left\{\mathcal{A}_{k}^{(t+1)}\right\}_{k=1}^{4}\right)$.

Therefore, by integrating the intermediate tensor reutilization  mechanism in Line 5 of Algorithm \ref{FCTNNLR} with the rapid generation of diverse combinations for $\operatorname{FCTN}\left(\left\{\mathcal{A}_{k}^{(t+1)}\right\}_{k=1}^{4}\right)$, we propose an unordered update FCTN factors mechanism to further accelerate Algorithm \ref{FCTNNLR}. The process of the unordered update FCTN factors is illustrated in Figure \ref{fig_6}.

\begin{figure}[!ht]
    \centering
    \includegraphics[width=1\textwidth,height=9cm]{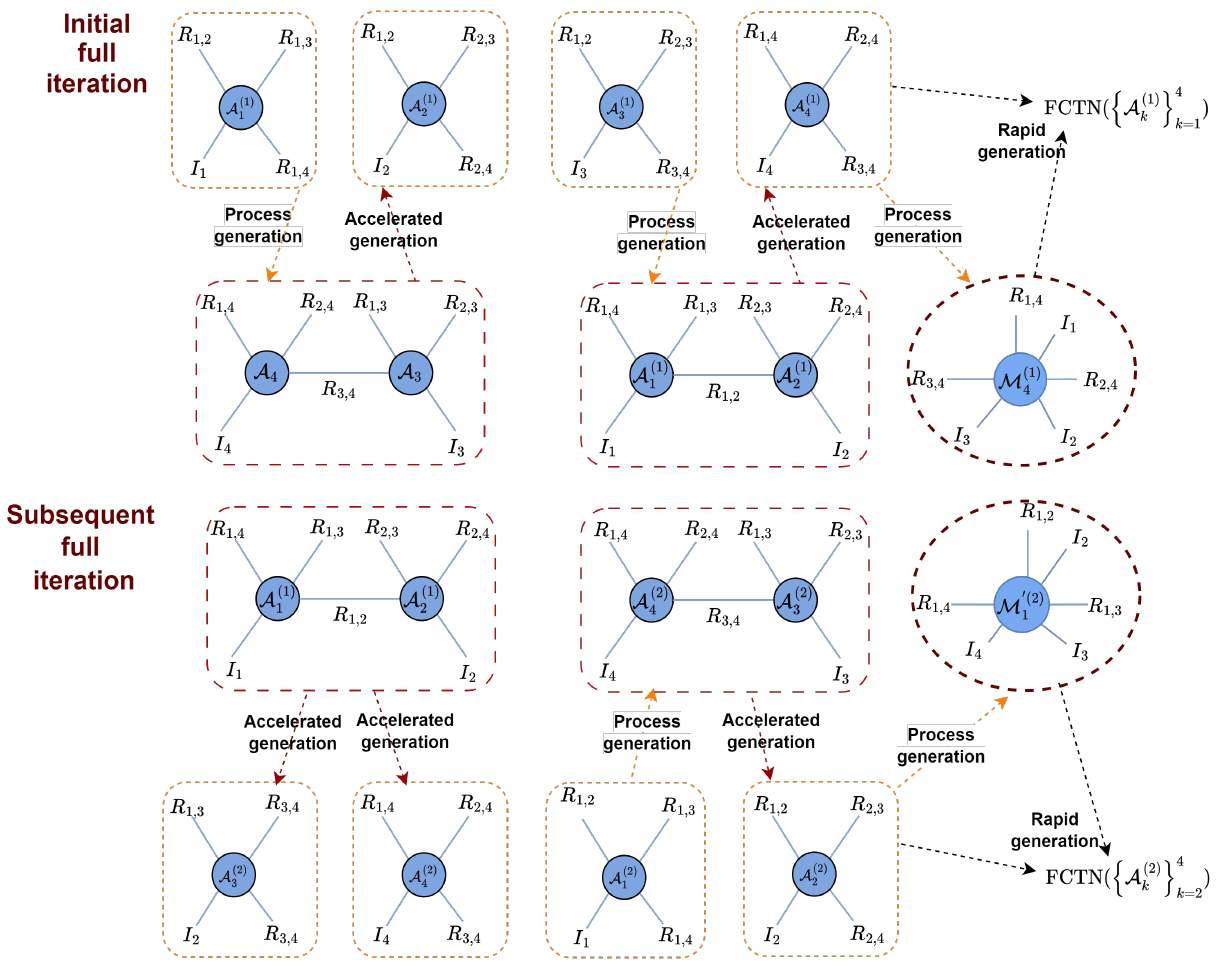}
    \caption{Unordered update of FCTN factors.}
    \label{fig_6}
\end{figure}

In accordance with the above acceleration mechanism, we propose the accelerated Algorithm \ref{FCTNNLR}, whose specific steps are presented below.

\begin{algorithm}[H]
\caption{PAM-AFCTNLR}\label{AFCTNNLR}
{\bfseries Input:} The observed tensor $\mathcal{Y}\in \mathbb{R}^{I_{1}\times \cdots \times I_{N}}$, the index of the known elements $\Omega$, the maximal FCTN-rank $r_{\text{FCTN}}^{max}$, and the parameters $\delta_k$ and $\lambda_k$  $\rho=0.1$, stopping criteria $\epsilon = 1\times10^{-4}$ and maximum number of iterations M.\\
{\bfseries Initialize:} The initial iteration index $t=1$, the initial FCTN-rank $\mathbf{r}_{\text{FNTC}}=\text{ones}(1,N(N-2)/2)$, $\mathcal{X}^{(1)}=\mathcal{Y}$, $\mathcal{A}_{k}^{(1)}=\text{rand}(R_{1,k},R_{2,k},\cdots,R_{k-1,k},I_{k},R_{k,k+1},\cdots,R_{k,N})$, periodically modified negative laplacian matrixs $\mathbf{L}_{k}$, iteration order 
$\mathbf{n}=1,2,\cdots,N$. Here $k=1,2,\cdots,N$ and elements generated by $\text{rand}(\cdot)$ follow a standard normal distribution.
\begin{algorithmic}[1]
\While{($\left \| \mathcal{X}^{(t+1)}-  \mathcal{X}^{(t)}\right \| _{F}/\left \| \mathcal{X}^{(t)} \right \|_{F}>\epsilon $) \textbf{ and } ($t =1,2,\cdots, M$)}
\For {$k=\mathbf{n}$}
 \State Obtain $\mathbf{X}_{(k)}^{(t)}$ via $\mathrm{Unfold}{(\mathcal{X}^{(t)})}$.
 \State Compute $\mathcal{M}_{k}^{(t)}$ or $\mathcal{M}_{k}^{'(t)}$ with \textbf{the intermediate tensor reuse mechanism}.
 \State Obtain $\mathbf{M}_k^{(t)}$ via ${\mathrm{GUnfold}(\mathcal{M}_k^{(t)},\mathbf{i}|\mathbf{n})}$.
 \State Compute $\mathbf{X}_{k}^{(t)}(\mathbf{M}^{(t)}_{k})^T$, $\operatorname{SVD}(\mathbf{M}^{(t)}_{k}(\mathbf{M}^{(t)}_{k})^T)$ and $\operatorname{FFT}(\mathbf{L}_k)$.
 \State Update $\mathcal{A}_k^{(t+1)}$ according to (\ref{eq_8}) and (\ref{eq_9}).
\EndFor
\State  \textbf{Update $\mathcal{X}^{(t+1)}$ efficiently with $\mathcal{M}_{k}^{'(t)}$ }.
\State  \textbf{Update $\mathbf{n}= \widetilde{\mathbf{n}}$, where $\widetilde{\mathbf{n}}$ is unordered permutation of $\mathbf{n}$, and let $t=t+1$}.
\EndWhile
\end{algorithmic}
{\bfseries Output:} The reconstructed tensor $\mathcal{X}$.
\end{algorithm}

\subsection{Computational complexity analysis}
In this subsection, we analyze the computational complexity of Algorithm \ref{FCTNNLR} and Algorithm \ref{AFCTNNLR} simultaneously, so as to compare the reduction in computational complexity achieved by Algorithm \ref{AFCTNNLR}  relative to Algorithm \ref{FCTNNLR}. We first analyze the computational complexity under the case of fourth-order tensor decomposition, and then extend it to the general $N$th order tensor case.

For ease of analysis, we assume that the input fourth-order tensor $\mathcal{Y}$ has the shape 
$I\times I \times I \times I $, and that its FCTN-rank is $(R_{1,2},R_{1,3},R_{1,4},R_{2,3},R_{2,4},R_{3,4}) \in \mathbb{R}^{6}$, where $R=R_{1,2}=R_{1,3}=\cdots=R_{3,4}$. The computational cost involved in Algorithm \ref{FCTNNLR} and Algorithm \ref{AFCTNNLR} consists of two parts: updating $\mathcal{A}_{k}$ and updating $\mathcal{X}$, where $k=1,2,3,4$.
First, updating $\mathcal{A}_{k}$ involves computing $\mathcal{M}_k$, matrix eigenvalue decomposition, one-dimensional FFT, matrix multiplication, component-wise multiplication, and component-wise division, where the computational costs of these operations are represented by computing $\mathcal{M}_k$, SVD, FFT, matrix multiplication, CWM, and CWD, respectively. Second, updating $\mathcal{X}$ is denoted by computing $\mathcal{X}$. With the above assumptions and notations, we derive the computational complexities of one full iteration of Algorithm \ref{FCTNNLR} and Algorithm \ref{AFCTNNLR} for fourth-order tensors, as shown in Table \ref{tab_1}.

\begin{table}[ht]
\centering
\fontsize{5}{15}\selectfont
\setlength{\abovecaptionskip}{0pt}
\setlength{\belowcaptionskip}{0pt}
\caption{Computational complexities of Algorithm \ref{FCTNNLR} and Algorithm \ref{AFCTNNLR} for fourth-order tensors.}\label{tab_1}
\begin{tabular}{c|c|c|c|c|c|c|c}
\hline
 $t=1$ &Computing $\mathcal{M}_k$ & SVD & FFT & Matrix multiplication & CWM & CWD & Computing $\mathcal{X}$ \\
\hline
Algorithm \ref{FCTNNLR} & $8(I^2+I^3)R^5$  & $\frac{16}{3}R^9$ &20$(I^2+IR^{3})logI$ & $8(I^3R^6+I^4R^3)$ &$4IR^3$   & $4IR^3$ & $2(I^2+I^3)R^5+2I^4R^3$ \\
\hline
Algorithm \ref{AFCTNNLR} & $4I^2R^5+8I^3R^5$ & $\frac{16}{3}R^9$ &20$(I^2+IR^{3})logI$ &$8(I^3R^6+I^4R^3)$  & $4IR^3$  & $4IR^3$  & $2I^4R^3$  \\
\hline
 $t\ge 2$ &Computing $\mathcal{M}_k$ & SVD & FFT & Matrix multiplication & CWM & CWD & Computing $\mathcal{X}$ \\
\hline
Algorithm \ref{FCTNNLR} & $8(I^2+I^3)R^5$  & $\frac{16}{3}R^9$ &20$(I^2+IR^{3})logI$ & $8(I^3R^6+I^4R^3)$ &$4IR^3$   & $4IR^3$ & $2(I^2+I^3)R^5+2I^4R^3$ \\
\hline
Algorithm \ref{AFCTNNLR} & $2I^2R^5+8I^3R^5$ & $\frac{16}{3}R^9$ &20$(I^2+IR^{3})logI$ &$8(I^3R^6+I^4R^3)$  & $4IR^3$  & $4IR^3$  & $2I^4R^3$  \\
\hline
\end{tabular}
\end{table}

From Table \ref{tab_1}, we can observe that the accelerated Algorithm \ref{AFCTNNLR} has lower computational complexity than Algorithm \ref{FCTNNLR} in both computing $\mathcal{M}_k$ and computing $\mathcal{X}$, with the reduction being particularly significant in the case of computing $\mathcal{X}$.

When extending to the case of an $N$th order tensor, we assume that the input tensor $\mathcal{Y}$ has the shape $I\times I \times \cdots \times I$ and its FCTN-rank is $(R,R,\cdots, R) \in \mathbb{R}^{\frac{N(N-1)}{2}}$. The remaining notations are consistent with those adopted in the analysis of the fourth-order tensor.

\begin{table}[ht]
\centering
\fontsize{5}{15}\selectfont
\setlength{\abovecaptionskip}{0pt}
\setlength{\belowcaptionskip}{0pt}
\caption{Computational complexities of Algorithm \ref{FCTNNLR} and Algorithm \ref{AFCTNNLR} for $N$th-order tensors.}\label{tab_2}
\begin{tabular}{c|c|c|c|c}
\hline
  &Computing $\mathcal{M}_k$ & SVD & FFT & Matrix multiplication \\
\hline
Algorithm \ref{FCTNNLR} & $ 2N{\textstyle \sum_{K=2}^{N-1}} I^kR^{k(N-k)+k-1}$  & $\frac{4}{3}N R^{3(N-1)}$ &$5N(I^2+IR^{(N-1)})logI$ & $2N(I^{N-1}R^{2(N-1)}+I^{N}R^{N-1})$  \\
\hline
Algorithm \ref{AFCTNNLR} & $4I^2R^{2N-3}+6I^3R^{3N-7}+$ & $\frac{4}{3}N R^{3(N-1)}$ &$5N(I^2+IR^{(N-1)})logI$ &$2N(I^{N-1}R^{2(N-1)}+I^{N}R^{N-1})$  \\
& $\cdots +2NI^{N-1}R^{2N-3}$ &   &  &\\
\hline
&  CWM & CWD & Computing $\mathcal{X}$ & \\
\hline 
Algorithm \ref{FCTNNLR}&  $NIR^{N-1}$& $NIR^{N-1}$ & $2{\textstyle \sum_{K=2}^{N}} I^kR^{k(N-k)+k-1}$  & \\
\hline 
Algorithm \ref{AFCTNNLR}& $NIR^{N-1}$ & $NIR^{N-1}$ & $2I^NR^{(N-1)} $ &  \\
\hline 
\end{tabular}
\end{table}

In Table \ref{tab_2}, where $k = 1, 2, \cdots, N$, the term $4I^2R^{2N-3}+6I^3R^{3N-7}+\cdots +2NI^{N-1}R^{2N-3}$ is given. If it contains terms identical to $I^{N-1}R^{2N-3}$, their coefficients are taken to be $2N$. For example, in the case of a fourth-order tensor, the computational complexity is $4I^2R^5+6I^3R^5$ according to the calculation of the first two terms. However, since the last term is identical to the second term, the computational complexity should be $4I^2R^5+8I^3R^5$. Under the intermediate tensor reutilization mechanism, Algorithm \ref{AFCTNNLR} significantly reduces the computational complexity of computing $\mathcal{M}$ and $\mathcal{X}$ for an $N$th order tensor, compared to Algorithm \ref{FCTNNLR}.

\subsection{Convergence analysis}
In this subsection, we analyze the theoretical convergence of the proposed Algorithm \ref{AFCTNNLR}.
For ease of analysis, we reformulate problem (\ref{model_2}) in the following form
\begin{equation}
f(\mathcal{A}_{1:N},\mathcal{X})=g(\mathcal{A}_{1:N},\mathcal{X})+\sum_{k=1}^{N}h_k(\mathcal{A}_k)+\iota_{\mathbb{D}}(\mathcal{X}),
\end{equation}
where $g(\mathcal{A}_{1:N},\mathcal{X})=\frac{1}{2}\left\|\mathcal{X}-\mathrm{FCTN}\left(\{\mathcal{A}_{k}\}_{k=1}^{N}\right)\right\|_{F}^{2}$, $h_k(\mathcal{A}_k)=\frac{\lambda_{k}}{2}\mathrm{tr}(\mathbf{A}^T_{k(k)}\mathbf{L}_k\mathbf{A}_{k(k)})$. We analyze the convergence guarantee of Algorithm \ref{AFCTNNLR} based on Theorem \ref{theorem_2}. 

\begin{theorem}\label{theorem_2}
    The sequence $\left \{ \mathcal{A}_{1:N}^{(t)},\mathcal{X}^{(t)}  \right \}_{s\in \mathbb{N} }$ generated by the Agorithm \ref{AFCTNNLR}, it converges to a critical point of $f(\mathcal{A}_{1:N},\mathcal{X})$.
\end{theorem}

\begin{proof}
According to Section 6 in \cite{Attouch2013}, to prove Theorem \ref{theorem_2}, we only need to verify that the following five conditions are satisfied.
   
1)The sequence $\left \{ \mathcal{A}_{1:N}^{(t)},\mathcal{X}^{(t)}  \right \}_{s\in \mathbb{N} }$ is bounded;

2)$f(\mathcal{A}_{1:N},\mathcal{X})$ has the  Kurdyka-Łojasiewicz (K-Ł) property at $\left \{ \mathcal{A}_{1:N}^{(t)},\mathcal{X}^{(t)}  \right \}_{s\in \mathbb{N} }$;

3) $g(\mathcal{A}_{1:N},\mathcal{X})+\sum_{k=1}^{N}h_k(\mathcal{A}_k)$ is a $C^1$ function whose gradient is Lipschitz contiuous, and $\iota_{\mathbb{D}}(\mathcal{X})$ is proper
lower semi-continuous functions;

4)\emph{Sufficient Decrease Condition}: For each $s\in \mathbb{N}$, there exists $\varphi \in (0,\infty )$ such the sequence $\left \{ \mathcal{A}_{1:N}^{(t)},\\ \mathcal{X}^{(t)} \right \}_{s\in \mathbb{N} }$ satisfies
$$
f(\mathcal{A}_{1:k}^{(t+1)},\mathcal{A}_{k+1:N}^{(t)},\mathcal{X}^{(t)})+\varphi | |\mathcal{A}_{k}^{(t+1)}- \mathcal{A}_{k}^{(t)}||_F^2 \le f(\mathcal{A}_{1:k-1}^{(t+1)},\mathcal{A}_{k:N}^{(t)},\mathcal{X}^{(t)}),
$$
$$
f(\mathcal{A}_{1:N}^{(t+1)},\mathcal{X}^{(t+1)})+\varphi ||\mathcal{X}_{k}^{(t+1)}- \mathcal{X}_{k}^{(t)}\|_F^2 \le f(\mathcal{A}_{1:N}^{(t+1)},\mathcal{X}^{(t)}),
$$
where $k=1,2,\cdots,N$;

5)\emph{Relative error condition}: For each $s\in \mathbb{N}$, there exists
$$\mathcal{B}_k^{(t+1)}\in \partial _{\mathcal{A}_k }f(\mathcal{A}_{1:k-1}^{(t+1)},\mathcal{A}_{k+1:N}^{(t)},\mathcal{X}^{(t)}),$$
$$
\mathcal{C}_k^{(t+1)} \in \partial _{\mathcal{X}} f(\mathcal{A}_{1:N}^{(t+1)},\mathcal{X}^{(t+1)}),
$$
and a constant $a\in (0,\infty )$ such that 
$$
||\mathcal{B}_k^{(t+1)}||_F \le a||\mathcal{A}_{k}^{(t+1)}- \mathcal{A}_{k}^{(t)}||_F,
$$
$$
||\mathcal{C}_k^{(t+1)}||_F \le a||\mathcal{X}_{k}^{(t+1)}- \mathcal{X}_{k}^{(t)}\|_F,
$$
where $k=1,2,\cdots,N$.

\emph{For condition (1)}: In Agorithm \ref{AFCTNNLR}, the elements of the initial $\mathcal{A}_{k}^{(1)}$ follow a standard normal distribution. Thus, the initial $\mathcal{A}_{k}^{(1)}$ is bounded. The $\mathcal{X}^{(1)}=\mathcal{Y}$ and the observed tensor $\mathcal{Y}$ is bounded, the initial is bounded. Assuming $\mathcal{A}_{k}^{(t)}$and $\mathcal{X}^{(t)}$ are both bounded, and satisfy $||\mathcal{A}_{k}^{(t)}||_F\le b$, $||\mathcal{X}^{(t)}||_F\le c$, where $b$ and $c$ are constants. Next, we prove that$||\mathcal{A}_{k}^{(t+1)}||_F$ and $||\mathcal{X}^{(t+1)}||_F$ are also bounded. According to the properties of the kronecker product, we can convert $\mathcal{A}_{k}^{(t+1)}$ into the following vector form
$$
\begin{aligned}
vec(\mathbf{A}_{k(k)}^{(t+1)})=&(\lambda_k\mathbf{I}_{s_k}\otimes\mathbf{L}_k+\mathbf{M}_{k}^{(t)}(\mathbf{M}^{(t)}_{k})^T\otimes\mathbf{I}_{q_k}+\rho\mathbf{I}_{s_kq_k})^{-1}vec(\mathbf{Y}^{(t)})\\
=&((\mathbf{C}\otimes\mathbf{F}^{\text{H}})(\lambda_k\mathbf{I}_{s_k}\otimes\Lambda+\Phi \otimes\mathbf{I}_{q_k}+\rho\mathbf{I}_{s_k}\otimes\mathbf{I}_{q_k})(\mathbf{C}^T\otimes\mathbf{F}))^{-1}vec(\mathbf{Y}^{(t)})\\
=&(\mathbf{C}\otimes\mathbf{F}^{\text{H}})(\lambda_k\mathbf{I}_{s_k}\otimes\Lambda+\Phi \otimes\mathbf{I}_{q_k}+\rho\mathbf{I}_{s_k}\otimes\mathbf{I}_{q_k})^{-1}(\mathbf{C}^T\otimes\mathbf{F})vec(\mathbf{Y}^{(t)}),
\end{aligned}
$$
where $\mathbf{Y}^{(t)}=\mathbf{X}^{(t)}_{(k)}\mathbf{M}^{(t)}_{k})^T+\rho\mathbf{A}^{(t)}_{k(k)}$.
Assume that the $i$-th element on the main diagonal of diagonal matrix $\lambda_k\mathbf{I}_{s_k}\otimes\Lambda+\Phi \otimes\mathbf{I}_{q_k}+\rho\mathbf{I}_{s_k}\otimes\mathbf{I}_{q_k}$ is $\varepsilon_i$ for $i=1,2,\cdots,s_kq_k$, then $\varepsilon_i \ge \rho >0$ as presented in Theorem 2 of reference \cite{Zheng2022}. Consequently, we have 
$$
\begin{aligned}
||\mathcal{A}_{k}^{(t+1)}||_F=&||vec(\mathbf{A}_{k(k)}^{(t+1)})||_F\\
=&||(\mathbf{C}\otimes\mathbf{F}^{\text{H}})(\lambda_k\mathbf{I}_{s_k}\otimes\Lambda+\Phi \otimes\mathbf{I}_{q_k}+\rho\mathbf{I}_{s_k}\otimes\mathbf{I}_{q_k})^{-1}(\mathbf{C}^T\otimes\mathbf{F})vec(\mathbf{Y}^{(t)})||_F\\
\le &||(\mathbf{C}\otimes\mathbf{F}^{\text{H}})(\lambda_k\mathbf{I}_{s_k}\otimes\Lambda+\Phi \otimes\mathbf{I}_{q_k}+\rho\mathbf{I}_{s_k}\otimes\mathbf{I}_{q_k})^{-1}(\mathbf{C}^T\otimes\mathbf{F})||_F ||vec(\mathbf{Y}^{(t)})||_F\\
=&||\lambda_k\mathbf{I}_{s_k}\otimes\Lambda+\Phi \otimes\mathbf{I}_{q_k}+\rho\mathbf{I}_{s_k}\otimes\mathbf{I}_{q_k})^{-1}||_F ||vec(\mathbf{Y}^{(t)})||_F\\
=&\sqrt{\sum_{i=1}^{s_kq_k}\frac{1}{\varepsilon_i^2 }} ||vec(\mathbf{Y}^{(t)})||_F \\
\le& \sqrt{\sum_{i=1}^{s_kq_k}\frac{1}{\rho^2 }} ||vec(\mathbf{Y}^{(t)})||_F\\
=&\frac{\sqrt{s_kq_k}}{\rho}(||\mathbf{X}^{(t)}_{(k)}\mathbf{M}^{(t)}_{k})^T+\rho\mathbf{A}^{(t)}_{k(k)}||_F)\\
\le &\frac{\sqrt{s_kq_k}}{\rho}(||\mathcal{X}^{(t)}||_F||\mathcal{M}_{k}^{(t)}||_F+||\mathcal{A}_k^{(t)}||_F)\\
\le&\frac{\sqrt{s_kq_k}}{\rho}(cb^{N-1}+\rho b).
\end{aligned}
$$
Therefore, for $k=1,2,\cdots,N$, $\mathcal{A}_{k}^{(t+1)}$ is bounded and we assume  $||\mathcal{A}_k^{(t+1)}||_F \le d$ , where $d$ is a constant. According to (\ref{eq_10}), we have 
$$
\begin{aligned}
||\mathcal{X}^{(t+1)}||_F=&||\mathcal{P}_{\Omega ^{c}}\left( \frac{\operatorname{FCTN}\left(\left\{\mathcal{A}_{k}^{(t+1)}\right\}_{k=1}^{N}\right)+\rho\mathcal{X}^{(t)}}{1+\rho}\right)+ \mathcal{P}_{\Omega}(\mathcal{Y})||_F \\
\le &||\mathcal{P}_{\Omega ^{c}}\left( \frac{\operatorname{FCTN}\left(\left\{\mathcal{A}_{k}^{(t+1)}\right\}_{k=1}^{N}\right)+\rho\mathcal{X}^{(t)}}{1+\rho}\right)||_F+ \mathcal{P}_{\Omega}(\mathcal{Y})||_F\\
\le & ||  \frac{\operatorname{FCTN}\left(\left\{\mathcal{A}_{k}^{(t+1)}\right\}_{k=1}^{N}\right)+\rho\mathcal{X}^{(t)}}{1+\rho}||_F +||\mathcal{Y}||_F \\
\le& \frac{||\operatorname{FCTN}\left(\left\{\mathcal{A}_{k}^{(t+1)}\right\}_{k=1}^{N}\right)||_F+\rho||\mathcal{X}^{(t)}||_F}{1+\rho}+||\mathcal{Y}||_F.
\end{aligned}
$$
As presented in Algorithm \ref{AFCTNNLR}, the $\operatorname{FCTN}\left(\left\{\mathcal{A}_{k}^{(t+1)}\right\}_{k=1}^{N}\right)$ can be formed by tensor contraction among FCTN factors $\mathcal{A}_1^{(t+1)}, \mathcal{A}_2^{(t+1)},\cdots, \mathcal{A}_N^{(t+1)}$. Thus, we have
$$
\begin{aligned}
||\mathcal{X}^{(t+1)}||_F \le& \frac{||\operatorname{FCTN}\left(\left\{\mathcal{A}_{k}^{(t+1)}\right\}_{k=1}^{N}\right)||_F+\rho||\mathcal{X}^{(t)}||_F}{1+\rho}+||\mathcal{Y}||_F \\
\le & \frac{ {\textstyle \prod_{k=1}^{N}||\mathcal{A}_{k}^{(t+1)}||_F} +\rho ||\mathcal{X}^{(t)} ||_F}{1+\rho}+||\mathcal{Y}||_F\le (d^N+\rho c)/(1+\rho)+ ||\mathcal{Y}||_F.
\end{aligned}
$$
Consequently, the $\mathcal{X}^{(t+1)}$ is bounded and condition (1) holds.

\emph{For condition (2)}: According to \cite{Attouch2013}, the semialgebraic functions has K-Ł property. The Frobenius norm function, trace function, and indicator function are  semialgebraic functions, and the sum of semialgebraic functions is also semialgebraic functions. Therefore, the function $f(\mathcal{A}_{1:N},\mathcal{X})$ is semialgebraic function and condition (2) holds.

\emph{For condition (3)}: We know that functions $g((\mathcal{A}_{1:N},\mathcal{X}))$ and $h_k(\mathcal{A}_k)$ are the Frobenius norm function and  trace function, respectively.
The Frobenius norm function and trace function are $C^1$ functions with Lipschitz continuous gradient, and $\iota_{\mathbb{D}}(\mathcal{X})$ as an indicator function is proper lower semi-continuous. Consequently, $g(\mathcal{A}_{1:N},\mathcal{X})+\sum_{k=1}^{N}h_k(\mathcal{A}_k)$ is a $C^1$ function with Lipschitz contiuous gradient, and $\iota_{\mathbb{D}}(\mathcal{X})$ is proper lower semi-continuous functions. Therefore, condition (3) holds.

\emph{For condition (4)}: The $\mathcal{A}_k^{(t+1)}$ is the optimal solution of the $\mathcal{A}_k^{(t)}$ subproblem, we have 
$$
\begin{aligned}
f(\mathcal{A}_{1:k}^{(t+1)},\mathcal{A}_{k+1:N}^{(t)},\mathcal{X}^{(t)})+&\frac{\varphi}{2} | |\mathcal{A}_{k}^{(t+1)}- \mathcal{A}_{k}^{(t)}||_F^2 \\
\le & f(\mathcal{A}_{1:k-1}^{(t+1)},\mathcal{A}_{k:N}^{(t)},\mathcal{X}^{(t)})+\frac{\varphi}{2} | |\mathcal{A}_{k}^{(t)}- \mathcal{A}_{k}^{(t)}||_F^2\\
=&f(\mathcal{A}_{1:k-1}^{(t+1)},\mathcal{A}_{k:N}^{(t)},\mathcal{X}^{(t)}),
\end{aligned}
$$
where $k=1,2,\cdots,N$. Similarly, the $\mathcal{X}^{(t+1)}$ is optimal the solution of the $\mathcal{X}_k^{(t)}$ subproblem, we have
$$
\begin{aligned}
f(\mathcal{A}_{1:N}^{(t+1)},\mathcal{X}^{(t+1)})+\varphi ||\mathcal{X}_{k}^{(t+1)}- \mathcal{X}_{k}^{(t)}\|_F^2 \le f(\mathcal{A}_{1:N}^{(t+1)},\mathcal{X}^{(t)}).
\end{aligned}
$$
Thus, condition (4) holds.

\emph{For condition (5)}: For each subproblem, we have
$$
\begin{array}{l}
0 \in \partial_{\mathcal{A}_{k}} f\left(\mathcal{A}_{1: k}^{(t+1)}, \mathcal{A}_{k+1: N}^{(t)}, \mathcal{X}^{(t)}\right)+\rho\left(\mathcal{A}_{k}^{(t+1)}-\mathcal{A}_{k}^{(t)}\right), \\
0 \in \partial_{\mathcal{X}} f\left(\mathcal{A}_{1: N}^{(t+1)}, \mathcal{X}^{(t+1)}\right)+\rho\left(\mathcal{X}_{k}^{(t+1)}-\mathcal{X}_{k}^{(t)}\right),
\end{array}
$$
where $k=1,2,\cdots,N$. Letting
$$
\begin{array}{l}
\mathcal{B}_{k}^{(t+1)}=-\rho\left(\mathcal{A}_{k}^{(t+1)}-\mathcal{A}_{k}^{(t)}\right) \in \partial_{\mathcal{A}_{k}} f\left(\mathcal{A}_{1: k}^{(t+1)}, \mathcal{A}_{k+1: N}^{(t)}, \mathcal{X}^{(t)}\right), \\
\mathcal{C}^{(t+1)}=-\rho\left(\mathcal{X}_{k}^{(t+1)}-\mathcal{X}_{k}^{(t)}\right) \in \partial_{\mathcal{X}} f\left(\mathcal{A}_{1: N}^{(t+1)}, \mathcal{X}^{(t+1)}\right),
\end{array}
$$
then we gain 
$$
\begin{array}{l}
\left\|\mathcal{B}_{k}^{(t+1)}\right\|_{F}=\left\|-\rho\left(\mathcal{A}_{k}^{(t+1)}-\mathcal{A}_{k}^{(t)}\right)\right\|_{F} \leq \rho\left\|\mathcal{A}_{k}^{(t+1)}-\mathcal{A}_{k}^{(t)}\right\|_{F}, \\
\left\|\mathcal{C}^{(t+1)}\right\|_{F}=\left\|-\rho\left(\mathcal{X}^{(t+1)}-\mathcal{X}^{(t)}\right)\right\|_{F} \leq \rho\left\|\mathcal{X}^{(t+1)}-\mathcal{X}^{(t)}\right\|_{F} .
\end{array}
$$
Therefore, condition (5) holds.
\end{proof}

\section{Numerical Experiments}\label{sect_4}
In this section, we evaluate the performance of the proposed AFCTNLR method. We compare our method with the following five methods under color videos and multi-temporal hyperspectral images, which are TMAC \cite{xu2015parallel}, TRLRF \cite{yuan2019tensor}, TRGFR \cite{wu2023tensor}, FCTN \cite{Zheng2021}, and FCTNFR \cite{Zheng2022} respectively. To assess the quality of the reconstructed results, we employ two evaluation metrics, the peak signal-to-noise ratio (PSNR) and the structural similarity index (SSIM). Their definitions are given as follows
$$
\operatorname{PSNR}=\frac{1}{I_{3} I_{4}} \sum_{i_{3}=1}^{I_{3}} \sum_{i_{4}=1}^{I_{4}} 10 \log _{10} \frac{\mathrm{Max}_{i_{3}, i_{4}}^{2}}{\frac{1}{I_{1} I_{2}}\left\|\mathbf{X}_{\mathrm{T}}^{i_{3}, i_{4}}-\bar{\mathbf{X}}^{i_{3}, i_{4}}\right\|_{F}^{2}}
$$
\noindent and
$$
\mathrm{SSIM}=\frac{1}{I_{3} I_{4}} \sum_{i_{3}=1}^{I_{3}} \sum_{i_{4}=1}^{I_{4}} \frac{\left(2 v\left(\mathbf{X}_{\mathrm{T}}^{i_{3}, i_{4}}\right) v\left(\bar{\mathbf{X}}^{i_{3}, i_{4}}\right)+c_{1}\right)\left(2 \omega\left(\mathbf{X}_{\mathrm{T}}^{i_{3}, i_{4}}, \bar{\mathbf{X}}^{i_{3}, i_{4}}\right)+c_{2}\right)}{\left(v\left(\mathbf{X}_{\mathrm{T}}^{i_{3}, i_{4}}\right)^{2}+v\left(\bar{\mathbf{X}}^{i_{3}, i_{4}}\right)^{2}+c_{1}\right)\left(\omega\left(\mathbf{X}_{\mathrm{T}}^{i_{3}, i_{4}}\right)^{2}+\omega\left(\bar{\mathbf{X}}^{i_{3}, i_{4}}\right)^{2}+c_{2}\right)},
$$
where $\mathcal{X}_{\mathrm{T}}\in \mathbb{R}^{I_1\times I_2 \times I_3 \times I_4}$ is the complete tensor, $\bar{\mathcal{X}}\in \mathbb{R}^{I_1\times I_2 \times I_3 \times I_4}$ is the reconstructed tensor. $\mathbf{X}^{i_3,i_4}_{\mathrm{T}} \in \mathbb{R}^{I_1 \times I_2}$ denotes a slice of $\mathcal{X}_{\mathrm{T}}$ obtained by fixing the indices $i_3$ and $i_4$, and $\bar{\mathbf{X}}_{i_3,i_4} \in \mathbb{R}^{I_1 \times I_2}$ denotes a slice of $\bar{\mathcal{X}}$ obtained similarly. $\text{Max}_{i_3,i_4}$ represents the maximum possible pixel value of $\mathbf{X}_{\text{T}}^{i_3,i_4}$ (here $\text{Max}_{i_3,i_4}=1$ because the pixel values of each test dataset are normalized to $[0,1]$). $\nu(\mathbf{X}_{\text{T}}^{i_3,i_4})$ and $\nu(\bar{\mathbf{X}}_{i_3,i_4})$ are the mean values of $\mathbf{X}_{\text{T}}^{i_3,i_4}$ and $\bar{\mathbf{X}}_{i_3,i_4}$, respectively. $\omega(\mathbf{X}_{\text{T}}^{i_3,i_4}, \bar{\mathbf{X}}_{i_3,i_4})$ denotes the covariance between $\mathbf{X}_{\text{T}}^{i_3,i_4}$ and $\bar{\mathbf{X}}_{i_3,i_4}$, while $\omega(\mathbf{X}_{\text{T}}^{i_3,i_4})$ and $\omega(\bar{\mathbf{X}}_{i_3,i_4})$ represent their standard deviations. $c_1$ and $c_2$ are constants.

All experiments were conducted in a Windows 10 and MATLAB environment, on a desktop computer equipped with a 12th Gen Intel(R) Core(TM) i7-12700 (2.10 GHz) processor and 16 GB of RAM.

\subsection{Color video} \label{sub_4.1}

In the first half of this subsection, we test six relatively old color videos\footnote{\url{http://trace.eas.asu.edu/yuv/}}, which are \emph{akiyo}, \emph{mother}, \emph{news}, \emph{salesman}, \emph{silent} and \emph{suzie} respectively. All these video datasets have a size of 144 $\times$ 176 $\times $ 3 $\times$ 50 (spatial height$\times$spatial width$\times$color channel$\times$frame number).  Our experiments are mainly conducted under the conditions of 5\%, 10\%, and 20\% sampling rates.

\emph{Parameter setting}: For a fourth-order tensor $\mathcal{X}^{I_1 \times I_2 \times I_3 \times I_4}$, the parameters that need to be considered for the proposed algorithm AFCTNLR include the maximal FCTN-rank $\mathbf{r}^{\text{max}}_{\text{FCTN}}=(R^{\text{max}}_{1,2}, R^{\text{max}}_{1,3} ,R^{\text{max}}_{1,4}, R^{\text{max}}_{2,3},$ $R^{\text{max}}_{2,4}, R^{\text{max}}_{3,4})$, the perturbation coefficient $\delta_k $ of the periodic perturbation matrix $\mathbf{L}_k$ and the regularization parameter $\lambda_k$,  where $k=1,2,3,4$. 
For the maximal FCTN-rank $\mathbf{r}^{\text{max}}_{\text{FCTN}}=(R^{\text{max}}_{1,2}, R^{\text{max}}_{1,3}, R^{\text{max}}_{1,4}, R^{\text{max}}_{2,3}, R^{\text{max}}_{2,4}, R^{\text{max}}_{3,4})$, we follow the configuration rules specified in reference \cite{Zheng2022}. Additionally, we set $\delta_k$ (for $k=1,2,3,4$) to the same value $\delta$, and $\lambda_k$ (for $k=1,2,3,4$) to the same value $\lambda$. Specifically, $\delta$ is selected from the candidate set $\left \{0.3, 0.4, 0.5, 0.6 \right \}$, while $\lambda$  is chosen from the candidate set $\left \{0.25, 0.35, 0.45, 0.55 \right \}$. 

In Tables \ref{tab_3}–\ref{tab_5}, we present the PSNR, SSIM, and running time (in seconds) of different algorithms on old color videos after 500 iterations, with the best results highlighted in bold. The choice of 500 iterations is made to ensure sufficient iterations for each algorithm, as their convergence conditions differ. We observe that among all the results, our proposed algorithm AFCTNLR achieves the best PSNR and SSIM values after 500 iterations, with improvements of nearly 1 dB in PSNR and 0.015 in SSIM over the FCTNFR algorithm in most cases. Although it does not have the shortest runtime, it reduces the running time by 10\%–15\% compared to the FCTNFR algorithm in most cases. To provide a more intuitive comparison of the reconstruction quality, we select the first frames of the reconstructed color videos \emph{salesman} and \emph{suize} under a sampling rate of 5\% for comparison, with the results shown in Figure (\ref{fig_7}).  The presented results demonstrate that our proposed algorithm AFCTNLR achieves superior image quality compared with the competing methods. Specifically, our algorithm is able to better preserve video details while recovering missing elements, as evidenced by the lower noise in the reconstructed images. In contrast, the reconstructed images from the competing methods still contain considerable noise, indicating a greater loss of local details during their reconstruction.

\begin{table}[!ht]
  \centering
  \fontsize{8}{9.25}\selectfont
  \setlength{\abovecaptionskip}{0pt}
  \setlength{\belowcaptionskip}{0pt}
  \caption{The performance on the old video with SR = 5\%} \label{tab_3}  
  \begin{tabular}{ccccccccc}  
    \hline  
    Video & Index & Observed & TMAC & TRLFR & TRGFR  & FCTN & FCTNFR & AFCTNLR  \\  
    \hline
     \emph{akiyo} & PSNR &7.3327  & 26.5866 &27.3712  &34.0454  &31.8108  &{35.3092}  &\textbf{36.7600}  \\  
          & SSIM &0.0159  & 0.7764 & 0.7896 & 0.9427 &0.9010  &{0.9618}  & \textbf{0.9743} \\ 
          & TIME(s)& & \textbf{53.8} & 229.6 & 736.9 &259.2 & 834.4 & 724.2  \\
    \emph{mother} & PSNR & 6.5874 & 25.4645 &28.4549  &34.5454  &32.7799  & {35.3139} &\textbf{36.4283}  \\  
          & SSIM & 0.0092 & 0.7066 & 0.7551 & 0.9089 &0.8534  & {0.9223} &\textbf{0.9402}  \\ 
          & TIME(s) &  & \textbf{53.3} & 234.1 &750.5 &268.3  &850.7& 758.2 \\
    \emph{news} & PSNR & 8.7488 & 23.8782 & 23.3870 & 29.1647 &27.0870  & {29.800} & \textbf{31.1363} \\  
          & SSIM & 0.0214 & 0.6714 & 0.6742 & 0.8648 & 0.7778 & {0.8852} & \textbf{0.9190} \\ 
          & TIME(s) &  & \textbf{46.8} & 223.4 & 700.4 & 200.3 & 821.2 & 774.5 \\
    \emph{salesman}& PSNR & 8.3215 & 22.7242 & 24.1370 & 29.5725 & 28.1512 &{30.2400}  & \textbf{31.7601} \\  
          & SSIM & 0.0147 &0.6885  & 0.6925 &0.8952  &0.8637  &{0.9189}  & \textbf{0.9376} \\ 
          & TIME(s) &  & \textbf{46.9} & 210.3 & 674.6 & 196.7 & 722.1 & 614.1 \\
    \emph{silent}& PSNR & 6.0500 &23.4658  &25.7317 & 31.9366 & 29.8580 & {33.1109} & \textbf{33.5041} \\  
          & SSIM & 0.0095 &0.6935  &0.7017  &0.9015  & 0.8455 & {0.9256} &\textbf{0.9350}  \\ 
          & TIME(s) &  & \textbf{54.2} & 254.2 & 816.8 & 288.5 & 973.7 & 875.2 \\
    \emph{suzie}& PSNR & 7.2362 & 22.7213 & 25.7910 & 31.0280 & 28.4288 & {30.9118} &\textbf{32.6744}  \\  
          & SSIM &  0.0100& 0.6140 &0.6944  & 0.8447 & 0.7274 & {0.8450} & \textbf{0.8926} \\ 
          & TIME(s) & & \textbf{46.8} & 248.3 & 743.1& 255.1 & 827.4 &717.8  \\
    \hline  
  \end{tabular}
\end{table}

\begin{table}[!ht]
  \centering
  \fontsize{8}{9.25}\selectfont
  \setlength{\abovecaptionskip}{0pt}
  \setlength{\belowcaptionskip}{0pt}
  \caption{The performance on the old video with SR = 10\%}\label{tab_4}    
  \begin{tabular}{ccccccccc}  
    \hline  
    Video & Index & Observed & TMAC & TRLFR & TRGFR  & FCTN & FCTNFR & AFCTNLR  \\  
    \hline
     \emph{akiyo} & PSNR &7.5609  & 33.4403 & 28.3448 & 37.2827 &36.2407  &{39.0102}  &\textbf{40.1701}  \\  
          & SSIM & 0.0249 & 0.9288 & 0.8267 &0.9694  &0.9585  &{0.9816}  &\textbf{0.9875}  \\ 
          & TIME(s) &  & \textbf{51.3} &232.3  & 749.4 & 282.4 &869.2  &763.7  \\
    \emph{mother} & PSNR &6.8220  & 34.3426 &29.3639  &36.7835  &36.7312  &{38.2051}  &\textbf{38.8727}  \\  
          & SSIM &0.0141  &0.9015  & 0.7895 & 0.9400 & 0.9294 &{0.9538}  &\textbf{0.9611}  \\ 
          & TIME(s) &  & \textbf{49.4} &218.4  & 675.6 &283.3  &750.4  &653.5  \\
    \emph{news} & PSNR &8.9848  & 26.7947 & 24.3281 & 31.2560 & 30.2554 & {32.9087} &\textbf{33.8056}  \\  
          & SSIM & 0.0336 & 0.8197 & 0.7169 & 0.9058 &0.8745  &{0.9362}  &\textbf{0.9500}  \\ 
          & TIME(s) &  & \textbf{51.7} & 234.6 & 752.5 & 238.7 & 831.0 & 733.4 \\
    \emph{salesman}& PSNR & 8.5564 & 28.1453 & 24.7456& 32.1631 & 31.4507 & {33.7308} & \textbf{34.7431} \\  
          & SSIM & 0.0268 &0.8508  & 0.7248 & 0.9369 &0.9283  &{0.9585}  & \textbf{0.9661} \\ 
          & TIME(s) &  & \textbf{63.7} & 222.2 &672.4  & 256.9 & 842.5 & 736.1 \\
    \emph{silent}& PSNR & 6.2854 & 29.9303 &  26.3893& 34.0167 & 33.6702 & {35.4642} & \textbf{36.3334} \\  
          & SSIM & 0.0153 & 0.8705 & 0.7320&0.9338  & 0.9246 & {0.9525} & \textbf{0.9615} \\ 
          & TIME(s) & &\textbf{63.0} & 237.0 & 783.1 & 297.2&922.7&792.5  \\
    \emph{suzie}& PSNR & 7.4703 & 27.1023 & 27.3146 &32.8535  & 31.9506 & {33.8984} & \textbf{35.0015} \\  
          & SSIM & 0.148 & 0.7565 & 0.7506 & 0.8839 & 0.8496 &{0.9064}  & \textbf{0.9284} \\ 
          & TIME(s) &  &\textbf{65.1}& 243.8 & 795.2 & 281.0 & 910.0 & 725.2 \\
    \hline  
  \end{tabular}
\end{table}

\begin{table}[!ht]
  \centering
  \fontsize{8}{9.25}\selectfont
  \setlength{\abovecaptionskip}{0pt}
\setlength{\belowcaptionskip}{0pt}
  \caption{The performance on the old video with SR = 20\%} \label{tab_5}  
  \begin{tabular}{ccccccccc}  
    \hline  
    Video & Index & Observed & TMAC & TRLFR & TRGFR  & FCTN & FCTNFR & AFCTNLR  \\  
    \hline
     \emph{akiyo} & PSNR &8.0779  & 34.8007 & 29.2302 & 39.7459  &40.7173  &{43.0130}  &\textbf{44.1559} \\  
          & SSIM & 0.0407 &0.9472 & 0.8528 &0.9811 &0.9827  &{0.9915}  &\textbf{0.9943}  \\ 
          & TIME(s) &  & \textbf{53.2} & 238.3 &782.2  &315.4&748.4 &591.5  \\
    \emph{mother} & PSNR & 7.3342 & 35.4304 & 30.5124 & 38.6087 & 39.5841 & {40.7932} & \textbf{41.6087} \\  
          & SSIM &  & 0.230 & 0.9226 & 0.8297 & 0.9617 &{0.9717}  &\textbf{0.9771}  \\ 
          & TIME(s) &  &\textbf{67.1}  &233.0  & 786.9 &  330.3& 882.3 & 796.2 \\
    \emph{news} & PSNR & 9.4953 & 28.2150 & 25.1692 &32.9193  & 32.8900 & {36.0432} &\textbf{36.5561}  \\  
          & SSIM &0.0655  &0.8694  & 0.7564 & 0.9319 &0.9262  & {0.9659} &\textbf{0.9714}  \\ 
          & TIME(s) &  &\textbf{65.2}  & 223.8 & 757.4& 259.0 &859.6  & 753.6 \\
    \emph{salesman}& PSNR & 9.0680 &29.2977  &25.5391  & 33.9426 & 34.1305 &{37.5741}  &\textbf{37.9341}  \\  
          & SSIM & 0.0513 & 0.8827 & 0.7628 & 0.9565 &0.9575  & {0.9804} & \textbf{0.9821} \\ 
          & TIME(s) &  & \textbf{70.6} &  251.4&859.9  & 304.4 & 960.6 & 759.2 \\
    \emph{silent}& PSNR & 6.7980 & 31.3768 & 27.7742& 35.8908 &36.6090  &{38.4584}  &\textbf{39.4025}  \\  
          & SSIM & 0.0256 &0.9043 & 0.7919 & 0.9541 & 0.9581 &{0.9736}  &\textbf{0.9790}  \\ 
          & TIME(s) &  &\textbf{70.8}  & 240.1 &832.2  & 328.4 & 924.0 & 825.5 \\
    \emph{suzie}& PSNR & 7.9820 & 28.6866 & 28.0657& 34.3378 &34.5330  &{36.3743}&\textbf{37.5446}   \\  
       & SSIM & 0.0226 &0.8152  & 0.7780 & 0.9116 & 0.9090 &{0.9415}  &\textbf{0.9524}  \\ 
        & TIME(s)&  & \textbf{70.1} & 240.5 & 802.0 & 292.2 & 910.3 &809.7  \\
    \hline 
  \end{tabular}
\end{table}

\begin{figure}[!ht]
\setlength\tabcolsep{0.5pt}
\centering
\begin{tabular}{cccccc}
\includegraphics[width=0.25\textwidth]{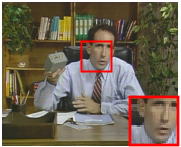}&
\includegraphics[width=0.25\textwidth]{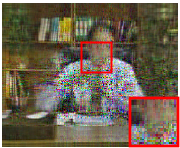}&
\includegraphics[width=0.25\textwidth]{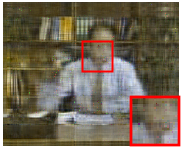}&
\includegraphics[width=0.25\textwidth]{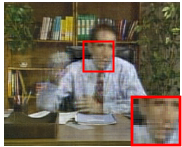}&\\
\footnotesize{Original image} & \footnotesize{TMAC}   &\footnotesize{TRLFR} &\footnotesize{TRGFR}  \\
\includegraphics[width=0.25\textwidth]{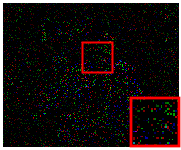}&
\includegraphics[width=0.25\textwidth]{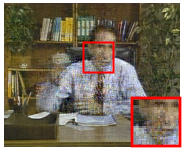}&
\includegraphics[width=0.25\textwidth]{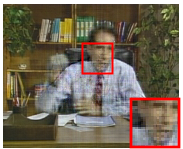}&
\includegraphics[width=0.25\textwidth]{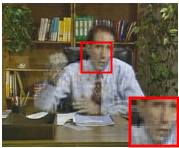}&\\
\footnotesize{Observed} & \footnotesize{FCTN}   & \footnotesize{FCTNFR} & \footnotesize{AFCTNLR}\\
\includegraphics[width=0.25\textwidth]{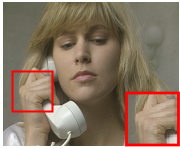}&
\includegraphics[width=0.25\textwidth]{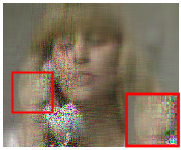}&
\includegraphics[width=0.25\textwidth]{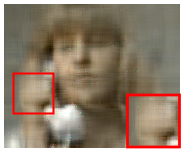}&
\includegraphics[width=0.25\textwidth]{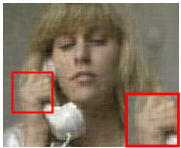}&\\
\footnotesize{Original image} & \footnotesize{TMAC}   &\footnotesize{TRLFR} &\footnotesize{TRGFR}  \\
\includegraphics[width=0.25\textwidth]{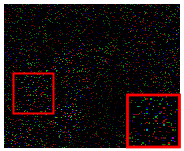}&
\includegraphics[width=0.25\textwidth]{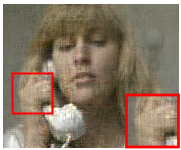}&
\includegraphics[width=0.25\textwidth]{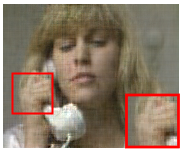}&
\includegraphics[width=0.25\textwidth]{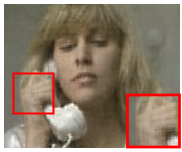}&\\
\footnotesize{Observed} & \footnotesize{FCTN}   & \footnotesize{FCTNFR} & \footnotesize{AFCTNLR}
\end{tabular}

\caption{Visualization of two color videos at a sampling rate (SR) of 5\%. The first two rows correspond to the first frame of the old color video \emph{salesman}, and the last two rows correspond to the first frame of the old color video \emph{suzie}.}\label{fig_7}
\end{figure} 

In the latter part of this subsection, we test three relatively new color videos, namely \emph{ukulele}\footnote{\url{https://www.pexels.com/video/serene-lady-playing-ukulele-in-scenic-hills-28769563/}}, \emph{flowers}\footnote{\url{https://www.pexels.com/video/close-up-of-blooming-cowslip-flowers-in-spring-29498903/}}, and \emph{train}\footnote{\url{https://www.pexels.com/video/stunning-aerial-view-of-swiss-lake-with-train-32781694/}}. Among them, the size of the \emph{ukulele} is 500 $\times$ 500 $\times $ 3 $\times$ 24 , the size of \emph{flowers} is 500 $\times$ 550 $\times $ 3 $\times$ 24, and the size of \emph{train} is 400 $\times$ 600 $\times $ 3 $\times$ 24. The parameter setting rules in this experiment are consistent with those in the preceding subsection.


In Tables \ref{tab_6}–\ref{tab_8}, we report the PSNR, SSIM, and running time (in seconds) of the algorithms on three larger color videos after 500 iterations, with the best results highlighted in bold. Consistent with the reconstruction results on the old color videos, our proposed algorithm still demonstrates superior performance on the new color videos, achieving the best PSNR and SSIM values compared with the competing methods. Moreover, our method shortens the running time by at least 24.9\% compared to the FCTNFR algorithm. To more clearly demonstrate the reconstructed images, we highlight the first frame of the color videos \emph{ukulele} and \emph{flowers } recovered by our algorithm at a 5\% sampling rate in Figure (\ref{fig_8}), and simultaneously present the recovery results of other algorithms for comparison. From the presented images, we can conclude that our proposed AFCTNLR algorithm achieves the best reconstruction results, as evidenced by the lower noise levels in the images it reconstructs compared with those from other algorithms, indicating that our method can preserve local details while reconstructing the overall image.

\begin{table}[!ht]  
  \centering
  \fontsize{9}{10}\selectfont
  \setlength{\abovecaptionskip}{0pt}
\setlength{\belowcaptionskip}{0pt}
  \caption{The performance of different algorithms is compared on video under SR = 5\%.}\label{tab_6} 
  \begin{tabular}{ccccccccc}  
    \hline  
    Video & Index & Observed & TMAC & TRLFR & TRGFR  & FCTN & FCTNFR & AFCTNLR  \\  
    \hline
     \emph{ukulele} & PSNR & 4.5272 & 30.0350 & 26.0787 &31.6583  &33.0136  &35.6269  &\textbf{36.4688}  \\  
          & SSIM &0.0079  &0.8062  &0.7440  &0.8616  &0.8902  &0.9408  &\textbf{0.9502}  \\ 
          & TIME(s) &  & \textbf{143.4} &652.2  &864.0  &724.2  &1505.8  &1130.6  \\
    \emph{flowers} & PSNR &8.8583  &31.4656  &27.2622  &33.5146  &36.1852  &{39.7142}  &\textbf{41.1405}  \\  
          & SSIM & 0.0865 &0.8377  &0.7511  &0.8808  &0.9110  &{0.9609}  &\textbf{0.9689}  \\ 
          & TIME(s) &  &\textbf{155.2}  &780.3  &1061.5  & 864.3 &2073.4  &1348.3  \\
    \emph{train} & PSNR &5.6282  &26.3704  &23.1579  &28.5771  &29.8533  &{32.6682}  & \textbf{33.4843} \\  
          & SSIM & 0.0098 & 0.7669 &0.6657  &0.8494  &0.8724  &{0.9347}  &\textbf{0.9439}  \\ 
          & TIME(s) &  &\textbf{134.9}  &667.7  &888.3  &726.7  &1851.6  &1192.5  \\
    \hline 
  \end{tabular}
\end{table}

\begin{table}[!ht]  
  \centering
  \fontsize{9}{10}\selectfont
  \setlength{\abovecaptionskip}{0pt}
\setlength{\belowcaptionskip}{0pt}
  \caption{The performance of different algorithms is compared on video under SR = 10\%.} \label{tab_7}  
  \begin{tabular}{ccccccccc}  
    \hline  
    Video & Index & Observed & TMAC & TRLFR & TRGFR  & FCTN & FCTNFR & AFCTNLR  \\  
    \hline
     \emph{ukulele} & PSNR & 4.7615 &31.7415  &27.6645  &32.7133  &35.3909  &38.2458  &\textbf{39.1923}  \\  
          & SSIM & 0.0117 & 0.8544 & 0.7727 & 0.8868 & 0.9344 & 0.9645 &\textbf{0.9704}  \\ 
          & TIME(s) &  & \textbf{146.7} & 612.8 &839.2  &914.4  &1859.0  &1189.2  \\
    \emph{flowers} & PSNR &9.0927  &33.0762  &27.9687  &34.4993  &38.7543  &{43.2623}  &\textbf{44.5751}  \\  
          & SSIM &0.0968  & 0.8715 & 0.7675 &0.8989  &0.9423  &{0.9788}  &\textbf{0.9830}  \\ 
          & TIME(s) &  & \textbf{152.3} & 712.4 &1080.0  & 966.1 &1922.2  &1276.8  \\
    \emph{train} & PSNR & 5.8628 & 28.0022 &24.3947  &29.4771  &32.5277  &{37.3527}  &\textbf{38.2206}  \\  
          & SSIM &0.0162  &0.8193  &0.7165  &0.8726  &0.9203  &{0.9648}  &\textbf{0.9744}  \\ 
          & TIME(s) &  & \textbf{127.3} &596.5  &780.3  &699.1  &1541.5  &1034.3  \\
    \hline 
  \end{tabular}
\end{table}

\begin{table}[!ht]  
  \centering
  \fontsize{9}{10}\selectfont
  \setlength{\abovecaptionskip}{0pt}
\setlength{\belowcaptionskip}{0pt}
  \caption{The performance of different algorithms is compared on video under SR = 20\%.}  \label{tab_8} 
  \begin{tabular}{ccccccccc}  
    \hline  
    Video & Index & Observed & TMAC & TRLFR & TRGFR  & FCTN & FCTNFR & AFCTNLR  \\  
    \hline
     \emph{ukulele} & PSNR & 5.2728 &32.5829  &28.4432  &33.5754  &37.4883  &40.8533  &\textbf{41.6743}  \\  
          & SSIM &0.0182  & 0.8780 & 0.7959 & 0.9038 &0.9561  & 0.9779 &\textbf{0.9813}  \\ 
          & TIME(s) &  & \textbf{178.3} &703.3  &1017.3  &1096.8  &2114.7  &1357.7  \\
    \emph{flowers} & PSNR &9.6045  &33.9615  &28.6585  &35.6328  &41.6912  &{46.0552}  &\textbf{47.3103}  \\  
          & SSIM & 0.1128 & 0.8916 &0.7751  &0.9168  &0.9659  &{0.9874}  &\textbf{0.9901}  \\ 
          & TIME(s) &  & \textbf{188.8} &796.5  &1140.1  &1119.1  &2288.4  &1458.2  \\
    \emph{train} & PSNR & 6.3742 &28.8667  & 25.1514 &30.4284  &34.6512  &{41.1696}  &\textbf{41.5797}  \\  
          & SSIM & 0.0288 &0.8471  &0.7465  &0.8945  &0.9430  &{0.9837}  &\textbf{0.9861}  \\ 
          & TIME(s) &  &\textbf{177.0}  &635.2  &857.8  &813.8  &1794.1  &1153.3  \\
    \hline 
  \end{tabular}
\end{table}

 
  
  
  
    
  
    
  
    
  
    

\begin{figure}[!ht]
\setlength\tabcolsep{0.5pt}
\centering
\begin{tabular}{cccccc}
\includegraphics[width=0.25\textwidth]{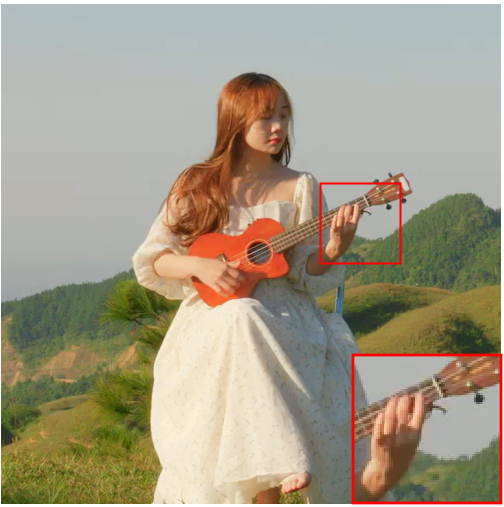}&
\includegraphics[width=0.25\textwidth]{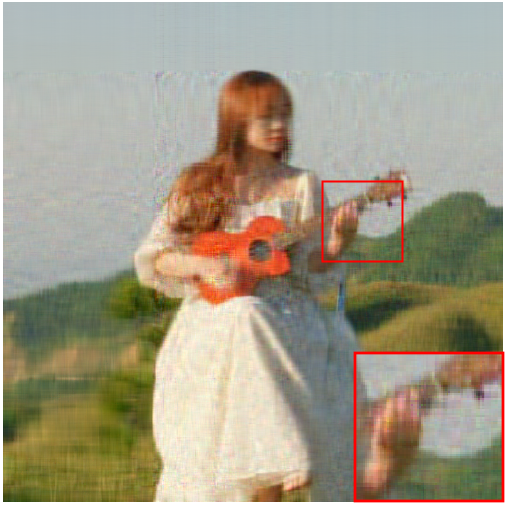}&
\includegraphics[width=0.25\textwidth]{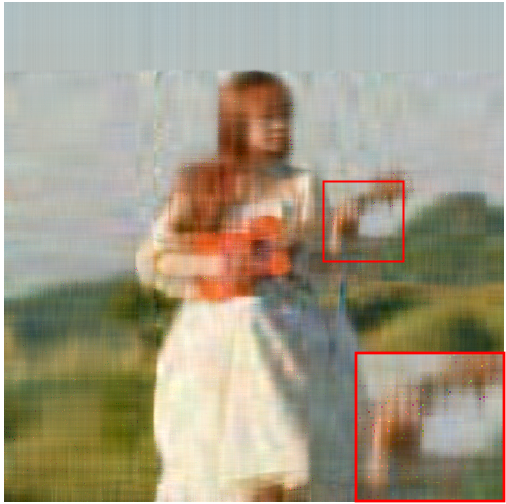}&
\includegraphics[width=0.25\textwidth]{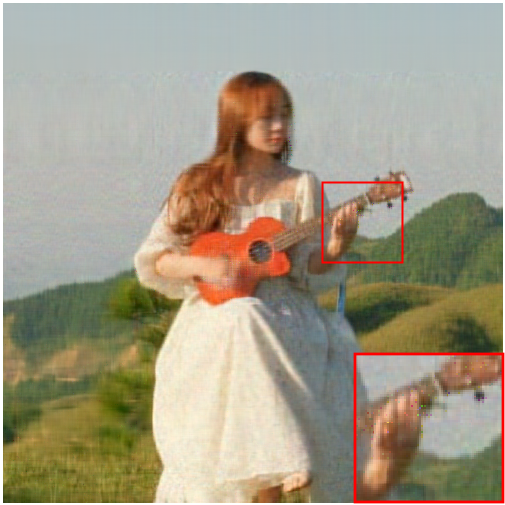}&\\
\footnotesize{Original image} & \footnotesize{TMAC}   &\footnotesize{TRLFR} &\footnotesize{TRGFR}  \\
\includegraphics[width=0.25\textwidth]{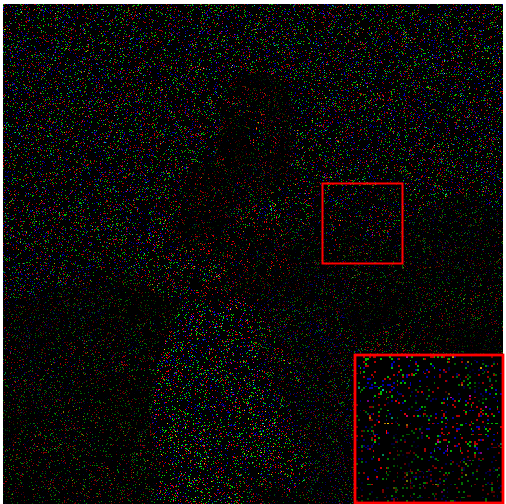}&
\includegraphics[width=0.25\textwidth]{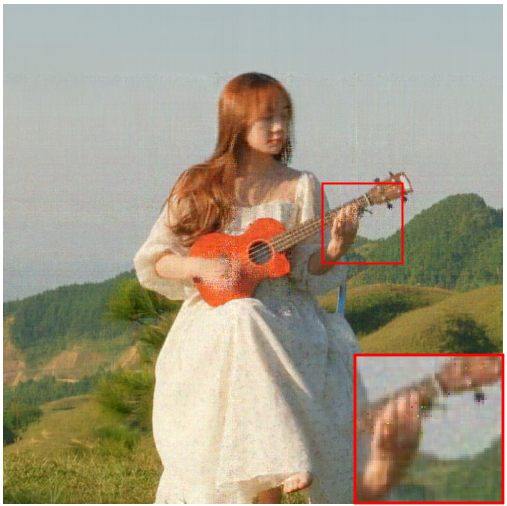}&
\includegraphics[width=0.25\textwidth]{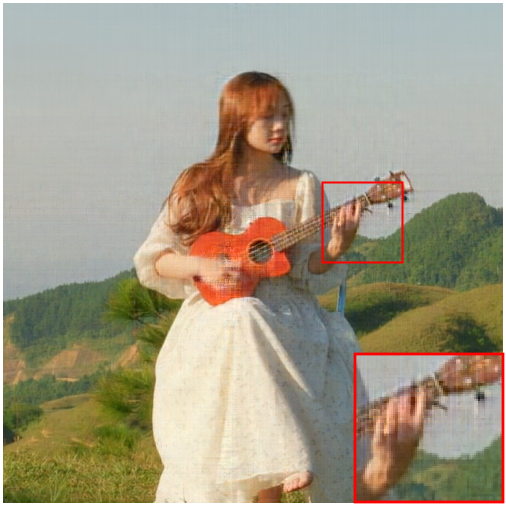}&
\includegraphics[width=0.25\textwidth]{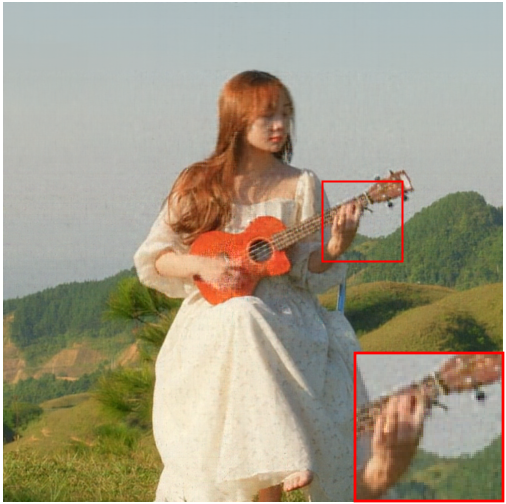}&\\
\footnotesize{Observed} & \footnotesize{FCTN}   & \footnotesize{FCTNFR} & \footnotesize{AFCTNLR}\\
\includegraphics[width=0.25\textwidth]{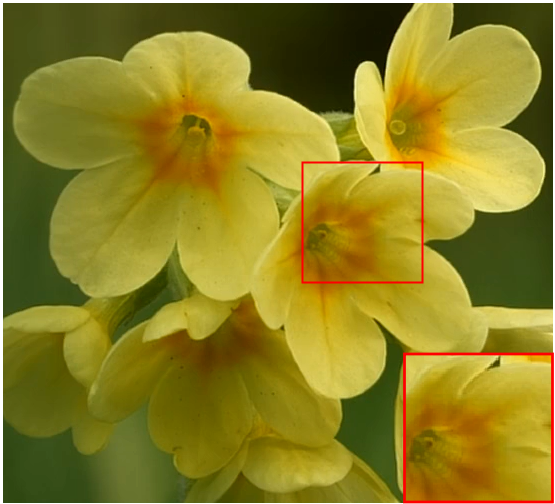}&
\includegraphics[width=0.25\textwidth]{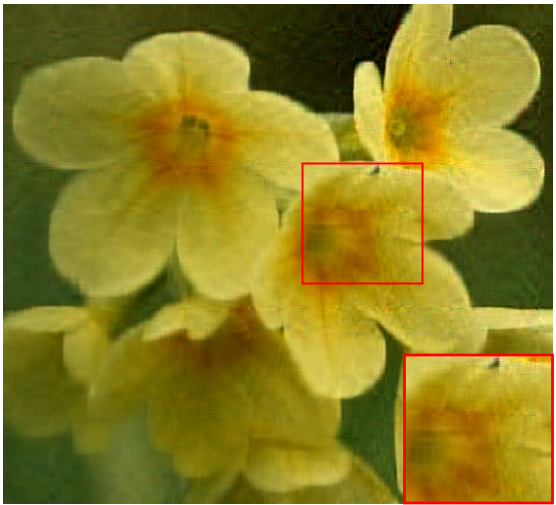}&
\includegraphics[width=0.25\textwidth]{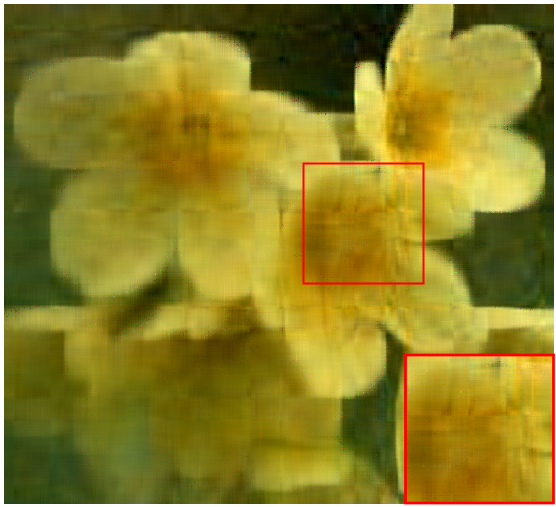}&
\includegraphics[width=0.25\textwidth]{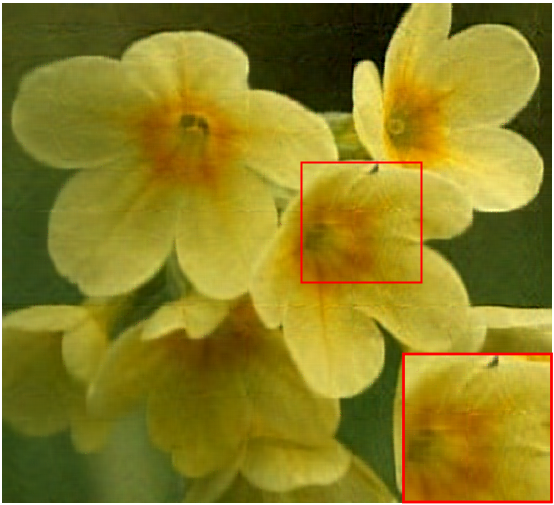}&\\
\footnotesize{Original image} & \footnotesize{TMAC}   &\footnotesize{TRLFR} &\footnotesize{TRGFR}  \\
\includegraphics[width=0.25\textwidth]{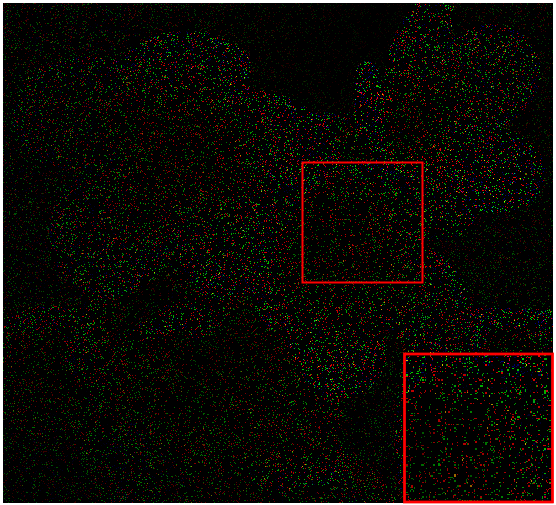}&
\includegraphics[width=0.25\textwidth]{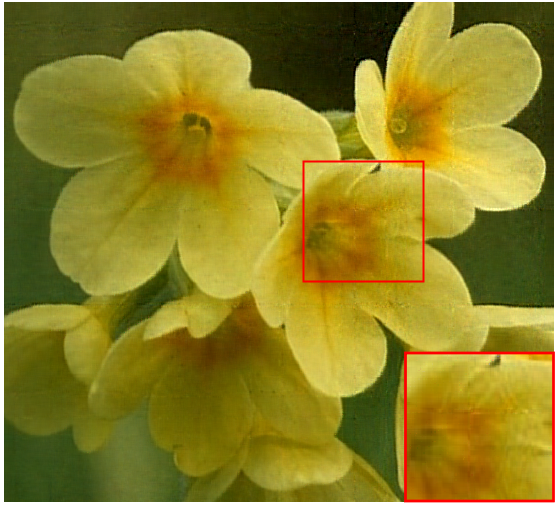}&
\includegraphics[width=0.25\textwidth]{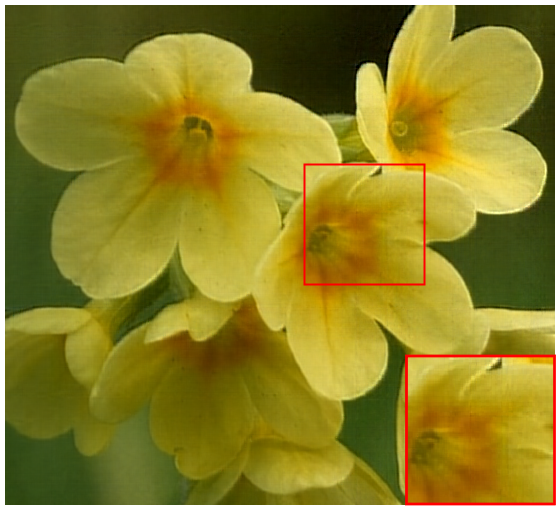}&
\includegraphics[width=0.25\textwidth]{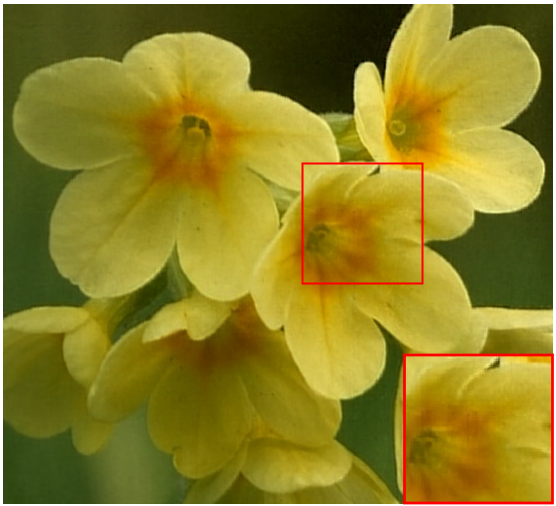}&\\
\footnotesize{Observed} & \footnotesize{FCTN}   & \footnotesize{FCTNFR} & \footnotesize{AFCTNLR}
\end{tabular}

\caption{Visualization of two color videos at a sampling rate (SR) of 5\%. The first two rows correspond to the first frame of the old color video \emph{ukulele}, and the last two rows correspond to the first frame of the old color video \emph{flowers}.}\label{fig_8}
\end{figure} 

\subsection{Multi-temporal hyperspectral image}
In this subsection, two multi-temporal hyperspectral images\footnote{\url{https://geodes.cnes.fr/}}, \emph{Brazil} and \emph{Morocco}, each of size $300\times300\times4\times5$ (spatial height $\times$ spatial width $\times$ band $\times$ time), which are cropped subsets of the original remote sensing images, are employed to validate the proposed method. Similar to the previous subsection, we consider the completion results under sampling rates of 5\%, 10\%.

To enhance the algorithm's performance on multi-temporal hyperspectral datasets, we introduce an asymmetric extrapolation acceleration mechanism to improve its computational efficiency.

\emph{asymmetric extrapolation \cite{xie2025efficient}}: This extrapolation is a straightforward generalization of momentum extrapolation, and its specific form is as follows 
\begin{equation}
\mathcal{A}_k^{(t+1)}=\mathcal{A}_k^{(t+1)}+\alpha(\mathcal{A}_k^{(t+1)}-\beta\mathcal{A}_k^{(t)}), 
\end{equation}
where $\alpha\in(0,1)$ and $\beta\in (0.2,0.8)$ represent tunable hyperparameters. The core idea of this extrapolation method is that when consecutive iterative solutions become overly similar, a momentum term related to the current iteration is introduced to generate solution-dependent perturbations, thereby the performance of the proposed algorithm. Unlike conventional momentum-based approaches, in which the extrapolation effect diminishes sharply when the solutions are too close, the proposed asymmetric momentum can introduce appropriate perturbations to ensure the acceleration effect of the algorithm.

\emph{Parameter setting}: 
We set the parameters of $\delta$ and $\lambda$ according to the same rules as those used in the subsection \ref{sub_4.1}. Meanwhile, the extrapolation coefficients $\alpha$ and $\beta$ are both set to 0.5. Regarding the setting of the maximal FCTN-rank $\mathbf{r}^{\text{max}}_{\text{FCTN}}$, we also follow the configuration rules specified in reference \cite{Zheng2022}.

\begin{table}[!ht]  
  \centering
  \fontsize{9}{10}\selectfont
  \setlength{\abovecaptionskip}{0pt}
\setlength{\belowcaptionskip}{0pt}
  \caption{Algorithm performance on multi-temporal hyperspectral images at SR = 5\%.}  \label{tab_9} 
  \begin{tabular}{ccccccccc}  
    \hline  
    Video & Index & Observed & TMAC & TRLFR & TRGFR  & FCTN & FCTNFR & AFCTNLR  \\  
    \hline
      \emph{Brazil} & PSNR & 19.0645 &27.4528  &32.5700 &36.0993  &32.5585  &37.0726 &\textbf{37.4943}  \\  
          & SSIM &0.0777  & 0.5304 & 0.7770 & 0.8656 &0.6933  & 0.8865 &\textbf{0.9011}  \\ 
          & TIME(s) &  & \textbf{30.4} &128.6  &212.7 &99.2  &154.9  &106.8 \\
    \emph{Morocco} & PSNR &14.6201 &24.4550  &33.0221  &37.1481  &32.0970  &{37.3686}  &\textbf{39.3587}  \\  
          & SSIM & 0.0567& 0.4256 &0.8244  &0.9101  &0.7209 &{0.9185}  &\textbf{0.9477}  \\ 
          & TIME(s) &  & \textbf{23.5} &132.2  &210.3 &241.0 &391.5  &264.3  \\
    \hline 
  \end{tabular}
\end{table}

\begin{table}[!ht] 
\vspace{-10pt}
  \centering
  \fontsize{9}{10}\selectfont
  \setlength{\abovecaptionskip}{0pt}
\setlength{\belowcaptionskip}{0pt}
  \caption{Algorithm performance on multi-temporal hyperspectral images at SR = 10\%.}  \label{tab_10} 
  \begin{tabular}{ccccccccc}  
    \hline  
    Video & Index & Observed & TMAC & TRLFR & TRGFR  & FCTN & FCTNFR & AFCTNLR  \\  
    \hline
      \emph{Brazil} & PSNR &19.8444 &34.2218  &34.2890  &37.8711  &36.1512  &{39.3744}  &\textbf{39.9345}  \\  
          & SSIM & 0.1167 & 0.7610 &0.8353  &0.9020  &0.8388 &{0.9240}  &\textbf{0.9360}  \\ 
          & TIME(s) &  & \textbf{28.4} &124.3  &208.4 &91.4 &119.9  &90.7 \\
    \emph{Morocco} & PSNR &14.8547  &34.3522  &34.7481 &39.1582  &38.5519 &{40.9691}  &\textbf{41.7584}  \\  
          & SSIM & 0.0752 & 0.8050 &0.8651  &0.9399  &0.9132  &{0.9596}  &\textbf{0.9668}  \\ 
          & TIME(s) &  & \textbf{26.8} &150.3  &220.6  &254.7  &429.0  &284.8  \\
    \hline 
  \end{tabular}
\end{table}

\begin{figure}[!ht]
\vspace{-5pt}
\setlength\tabcolsep{0.5pt}
\centering
\begin{tabular}{cccccc}
\includegraphics[width=0.25\textwidth]{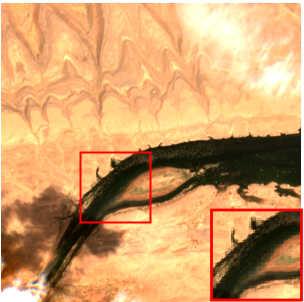}&
\includegraphics[width=0.25\textwidth]{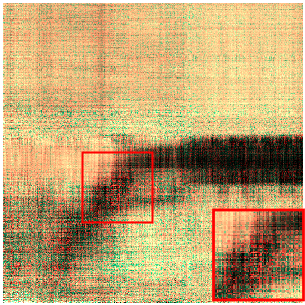}&
\includegraphics[width=0.25\textwidth]{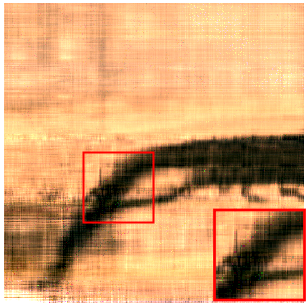}&
\includegraphics[width=0.25\textwidth]{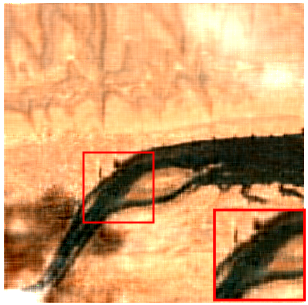}&\\
\footnotesize{Original image} & \footnotesize{TMAC}   &\footnotesize{TRLFR} &\footnotesize{TRGFR}  \\
\includegraphics[width=0.25\textwidth]{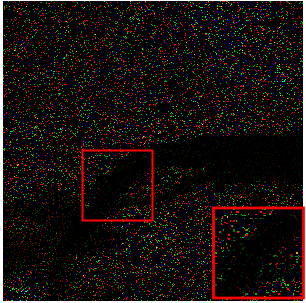}&
\includegraphics[width=0.25\textwidth]{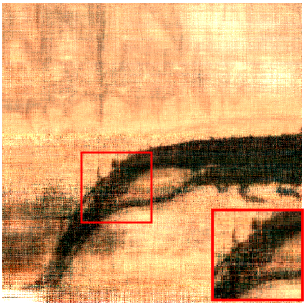}&
\includegraphics[width=0.25\textwidth]{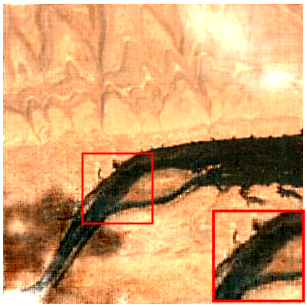}&
\includegraphics[width=0.25\textwidth]{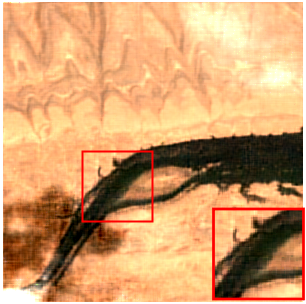}&\\
\footnotesize{Observed} & \footnotesize{FCTN}   & \footnotesize{FCTNFR} & \footnotesize{AFCTNLR}\\
\end{tabular}

\caption{The visualization of the multi-temporal hyperspectral image \emph{Morocco} at SR = 5\% is presented, with bands 3, 2, and 1 at the third time point mapped to the red, green, and blue channels to form an RGB image.
}\label{fig_9}
\end{figure} 

In Tables \ref{tab_9}–\ref{tab_10}, we present the PSNR, SSIM, and running time (in seconds) of our evaluated algorithms on multi-temporal hyperspectral images after 500 iterations, with the best results highlighted in bold. We perform completion with $R^{max}_{1,3}=3$ on dataset  \emph{Brazil} and $R^{max}_{1,3}=4$ on dataset \emph{Morocco}, so it can be seen from the tables that the last three algorithms exhibit significant differences in running time on these two datasets. As indicated in the table, the proposed AFCTNLR algorithm attains the highest PSNR and SSIM values among all competing methods, while its runtime is at least 23\% lower than that of FCTNFR, which achieves the second-best PSNR and SSIM. In Figure \ref{fig_9}, we present the spatial images at the third time point of the multi-temporal hyperspectral image \emph{Morocco}, reconstructed from the B4, B3, and B2 bands of the remote sensing images. We observed that the proposed AFCTNLR method exhibits a clear superiority over alternative algorithms, achieving more accurate image reconstruction and delivering better performance in detail preservation.

 
  
  

    
  
    

\section{Conclusion}\label{sect_5}
In this paper, we proposed a novel tensor completion model that combines trace regularization with FCTN decomposition. The proposed method leveraged a periodically modified negative laplacian together with FCTN factor matrices under a trace function framework. By constraining the second-order derivatives of discrete factors, it strengthened their smoothness and indirectly ensured a more continuous reconstructed tensor. We developed a PAM-based algorithm to solve the proposed model and proved its convergence theoretically. Meanwhile, we designed an intermediate tensor reuse mechanism, which effectively reduced the algorithm’s running time by 10\%–30\% without compromising the quality of image reconstruction, while complexity analysis confirmed its effectiveness in lowering computational costs. Comprehensive experiments revealed that the proposed AFCTNLR algorithm surpassed competing approaches in terms of reconstruction performance.


\section*{Acknowledgment}
\noindent{\bf Funding:} This research is supported by the National Natural Science Foundation of China (NSFC) grants 92473208, 12401415, the Key Program of National Natural Science of China 12331011, the 111 Project (No. D23017),  the Natural Science Foundation of Hunan Province (No. 2025JJ60009). 

\noindent{\bf Data Availability:} Enquiries about data/code availability should be directed to the authors.

\noindent{\bf Competing interests:} The authors have no competing interests to declare that are relevant to the content of this paper.

   

\appendix

\bibliographystyle{plain}	
\bibliography{refs}

\end{document}